\documentclass[a4paper]{article}
\usepackage[all]{xy}
\usepackage[dvips]{graphics,graphicx}
\usepackage{amsfonts,amssymb,amsmath,color,mathrsfs, amstext}
\usepackage{amsbsy, amsopn, amscd, amsxtra, amsthm,authblk}
\usepackage{enumerate,algorithmicx,algorithm}
\usepackage{upref}
\usepackage{geometry}
\usepackage{lineno}
\usepackage{float}
\usepackage{yhmath}
\usepackage{booktabs}
\usepackage{subfigure}
\usepackage{multirow}
\usepackage{makecell}
\usepackage{ulem}
\usepackage[colorlinks,
            linkcolor=red,
            anchorcolor=red,
            citecolor=red
            ]{hyperref}
\usepackage{appendix}

\def\b{\boldsymbol}
\def\E{\mathbb{E}}
\def\R{\mathbb{R}}
\def\P{\mathbb{P}}
\def\T{\mathbb{T}}
\def\cL{\mathcal{L}}
\def\cH{\mathcal{H}}

\newtheorem{theorem}{Theorem}[section]
\newtheorem{lemma}{Lemma}[section]
\newtheorem{remark}{Remark}[section]

\newtheorem{assumption}{Assumption}[section]
\newtheorem{proposition}{Proposition}[section]
\newtheorem{definition}{Definition}[section]

\numberwithin{equation}{section}

\title{A second-order Langevin sampler preserving positive volume for isothermal–isobaric ensemble}
\date{}

\author[1,2,3]{Lei Li\thanks{E-mail: leili2010@sjtu.edu.cn}}
\author[1]{Yuzhou Peng\thanks{E-mail: pengyz@sjtu.edu.cn}}
\affil[1]{School of Mathematical Sciences, Shanghai Jiao Tong University, Shanghai, 200240, P. R. China}
\affil[2]{Institute of Natural Sciences and MOE-LSC, Shanghai Jiao Tong University, Shanghai, 200240, P. R. China}
\affil[3]{Shanghai Artificial Intelligence Laboratory, Shanghai, P. R. China}

\begin{document}

\maketitle

\begin{abstract}
We propose in this work a second-order Langevin sampler for the isothermal-isobaric ensemble (the NPT ensemble), preserving a positive volume for the simulation box. We first derive the suitable equations of motion for particles to be coupled with the overdamped Langevin equation of volume by sending the artificial mass of the periodic box to zero in the work of Liang et. al. [J. Chem. Phys. 157(14)]. We prove the well-posedness of the new system of equations and show that its invariant measure is the desired ensemble. The new continuous time equations not only justify the previous cell-rescaling methods, but also allow us to choose a suitable friction coefficient so that one has additive noise after a change of variable by taking logarithm of the volume. This observation allows us to propose a second order weak scheme that guarantees the positivity of the volume.
Various numerical experiments have been performed to demonstrate the efficacy of our method. 
\end{abstract}

\allowdisplaybreaks

\section{Introduction}

Molecular dynamics (MD) simulation has made a great contribution to the study of the thermodynamic properties of complex multi-particle systems across various disciplines, including physics, chemistry, biology and pharmaceuticals \cite{allen2017computer, brunger1987crystallographic, hollingsworth2018molecular, karplus1990molecular, yamakov2002dislocation}. 
Given the large or infinite number of particles in these macro-scale systems, periodic boxes are typically utilized to approximate the entire system, dividing the whole space into infinitely many adjacent simulation boxes, with one being the original simulation box and others being copies called images. Periodic boundary conditions enable particles leaving one side of the simulation box to re-enter from the opposite side with the same momenta, ensuring that all the particles within the box experience a similar environment. The application of periodic boxes facilitates the study of larger systems with fewer particles, reducing computational costs while still providing accurate insights into the behavior of the system as a whole \cite{frenkel2023understanding}. In practice, these systems are often subjected to constant temperature and/or constant pressure conditions. This work focuses on the isothermal–isobaric ensemble, where the monatomic particle system interacts with an external bath that maintains constant temperature and pressure throughout the simulation. This ensemble, commonly known as the NPT ensemble because the number of particles, pressure and temperature are all kept constant, reflects laboratory conditions typical in simulations of solvated proteins, membranes and viruses \cite{tobias1997atomic}.

Various thermostats and barostats have been developed to modeled constant temperature and pressure in the external baths. The first barostat for maintaining pressure during the MD simulation was proposed in the pioneering work of Andersen \cite{andersen1980molecular}, where the volume of the simulation box fluctuates based on the difference between external and internal pressure. The thermostat for preserving temperature works through stochastic collisions modeled by momenta resampling. Subsequently, Parrinello and Rahman extended Andersen's barostat to accommodate periodic boxes of arbitrary symmetry in the isoenthalpic-isobaric ensemble \cite{parrinello1980crystal,parrinello1981polymorphic}. Later, Andersen's technique was extended to deal with rigid molecule systems, where the intermolecular potential is based on atom–atom interactions \cite{ryckaert1983introduction} and the periodic box could be trigonal \cite{nose1983structural, nose1983study}. However, Andersen's method lacks effectiveness in preserving the dynamical properties of the systems.

An alternative approach to controlling temperature and pressure was proposed by Berendsen et al., where the momenta of the particles and the volume of the simulation box fluctuate according to the difference between the instantaneous value of temperature and pressure and their desired counterparts \cite{berendsen1984molecular}. Although efficient in equilibrating the system, Berendsen's barostat inaccurately samples the target ensemble and is typically employed as a burn-in technique \cite{prasad2018best}. Motivated by the stochastic version of the Berendsen thermostat \cite{bussi2007canonical, bussi2009isothermal}, a carefully chosen stochastic term was incorporated into the Berendsen barostat, leading to the development of the stochastic cell rescaling method \cite{bernetti2020pressure}. However, concerns remain regarding its algorithmic specifics and fidelity to the invariant measure under typical discretizations such as the commonly used Euler-Maruyama scheme.

The Nos{\'e}–Hoover thermostat \cite{hoover1985canonical,nose1984molecular,nose1984unified,nose1983constant} replaces the stochastic collisions in Andersen's thermostat with an auxiliary variable to simulate the effects of the bath, ensuring smooth, deterministic and time-reversible trajectories and performing better in preserving the dynamical properties of the system. However, it is inefficient in equilibrating the system due to the lack of ergodicity \cite{frenkel2023understanding}. The Nos{\'e}–Hoover chain method was then proposed to solve this problem by introducing additional parameters \cite{martyna1992nose}. Later, the Martyna-Tobias-Klein (MTK) algorithm, combining Nos{\'e}-Hoover thermostats with suitable barostats, and its subsequent extensions have found widespread use in NPT ensemble simulations \cite{martyna1994constant, martyna1996explicit}.

Another widely adopted thermostat is the Langevin dynamics \cite{frenkel2023understanding}, which employs dissipative forces and noise satisfying the fluctuation-dissipation relation. A significant advantage of the Langevin dynamics is the ergodicity. The Langevin piston algorithm and its extensions \cite{bussi2007accurate, di2015stochastic, feller1995constant, gao2016sampling, gronbech2014constant, kolb1999optimized} combine the Langevin thermostat with suitable barostats to maintain constant temperature and pressure. 

In this work, we consider a system of $N$ particles with masses $\{m_i\}_{i=1}^N$, positions $\{\b{r}_i\}_{i=1}^N$ and momenta $\{\b{p}_i\}_{i=1}^N$, contained within a cubic periodic box of length $L$ and volume $V=L^3$. The particles interact with each other through some certain potential under constant pressure $P_0$ and temperature $T_0$. For compactness, we fix $N$, set $d=3N$ and define
\begin{gather}
\begin{split}
&\b{r}=(\b{r}_1,\cdots, \b{r}_N)\in \T_L^d,\\
& \b{p}=(\b{p}_1,\cdots, \b{p}_N)\in \R^d,\\
& \b{m}=\text{diag}(m_1\b{I}_3,\cdots, m_N\b{I}_3), \\
\end{split}
\end{gather}
where $\T_L^d$ represents a $d$-dimensional torus with side length $L$. The potential energy of the system is denoted by $U(\b{r}; V)$, indicating the dependence on the volume of the simulation box, and the stationary distribution of the NPT ensemble is given by
\begin{gather}\label{eq:invariant_measure}
    \pi(d\b{r}, d\b{p}, dV) \propto \exp{\left[-\beta\left(U(\b{r}; V) + \frac{1}{2}\b{p}^T\b{m}^{-1}\b{p} + P_0V\right)\right]}d\b{r}d\b{p}dV, 
\end{gather}
where $\beta = (k_BT_0)^{-1}$ with $k_B$ denoting the Boltzmann constant and $T_0$ the temperature. 
An important statistics to evaluate the performance of our algorithm is the instantaneous pressure of the system, given by
\begin{gather}\label{eq:inspressure1}
P = \frac{1}{3V}\b{p}^T\b{m}^{-1}\b{p}
-\frac{1}{3V}\b{r}^T\nabla_{\b{r}}U(\b{r}; V)-\frac{\partial }{\partial V}U(\b{r}; V).
\end{gather}
There are two important pressure virial theorems associated with the NPT ensemble and both of them are invariant properties of the NPT ensemble. 
\begin{gather}\label{eq:P-virial}
    \begin{split}
        \langle P\rangle &= P_0,\\
        \langle PV\rangle &= P_0\langle V\rangle - \beta^{-1},
    \end{split}    
\end{gather}
where $\langle \cdot \rangle$ denotes the NPT ensemble average. The first theorem relates the internal and external pressure, ensuring that the constant pressure condition is satisfied. The second theorem relates the internal and external work, which is critical for verifying the equilibrium and correctness of molecular dynamics simulations in NPT conditions. Interested readers can refer to \cite[Appendix B]{martyna1994constant} for the verification of these two theorems and further details.

By introducing an artificial mass $M$ and a momentum variable $p^V$ for the cell, the following equations of motion has been derived in \cite[Eq. (3.13)]{liang2022random}, which can give the target NPT ensemble:
\begin{gather}\label{eq:liang313}
    \begin{split}
        & \dot{\b{r}} = \b{m}^{-1}\b{p} + \frac{\dot{V}}{3V}\b{r}, \,\dot{\b{p}} = -\nabla_{\b{r}} U - \frac{\dot{V}}{3V}\b{p} - \b{\gamma} \b{p} + \sqrt{2\beta^{-1}\b{\gamma}\b{m}}\, \dot{\b{W}},\\
        &\dot{V} = \frac{p^V}{M}, \, \dot{p}^V = -\frac{\partial H}{\partial V} - \tilde{\gamma}(V)p^V + \sqrt{2\beta^{-1}\tilde{\gamma}(V)M}\,\dot{W}_V,
    \end{split}
\end{gather}
where $\b{\gamma} = \text{diag}(\gamma_1\b{I}_3,\cdots, \gamma_N\b{I}_3)$ and $\tilde{\gamma}(V)$ are the friction coefficients of $\b{p}$ and $p^V$, respectively, each component of $\b{W}\in\R^d$ and $W_V\in\R$ are independent standard Wiener processes.
The Hamiltonian (see \cite{liang2022random}) is given by
\begin{gather}
H(\b{s},\b{p}^{\b{s}},V, p^V) = U(V^{1/3}\b{s}; V) + P_0V + \frac{1}{2}V^{-2/3}\b{p}^{\b{s},T}\b{m}^{-1}\b{p}^{\b{s}} + \frac{(p^V)^2}{2M},
\end{gather}
where $\b{s} = V^{-1/3}\b{r}$ is the reduced variable and $\b{p}^{\b{s}} = V^{1/3}\b{p}$ is the conjugate variable of $\b{s}$.
With this Hamiltonian, the Langevin system \eqref{eq:liang313} can be written as
\begin{gather}\label{eq:originalunderdampedLang}
    \begin{split}
        d\b{s} & = \frac{\partial H}{\partial\b{p}^{\b{s}}} dt,\, d\b{p}^{\b{s}} = -\frac{\partial H}{\partial\b{s}} dt - V^{2/3}\b{\gamma m} \frac{\partial H}{\partial\b{p}^{\b{s}}} dt + V^{1/3}\sqrt{2\beta^{-1}\b{\gamma}\b{m}}\, d\b{W},\\
        dV & =\frac{\partial H}{\partial p^V}dt,\, dp^V = -\frac{\partial H}{\partial V}dt - M\tilde{\gamma}(V)\frac{\partial H}{\partial p^V}dt + \sqrt{2\beta^{-1}\tilde{\gamma}(V)M}\,dW_V.
    \end{split}
\end{gather}
We remark that the instantaneous pressure \eqref{eq:inspressure1} satisfies $\frac{\partial H}{\partial V} = P_0-P$.

Note that the artificial mass $M$ of the volume has no actual physical meaning, and thus we are concerned with the zero-mass limit $M\to0$, which makes more sense and is more convenient for simulations. It has been verified in \cite{liang2022random} that the limit equation for $V$ coincides with the stochastic cell rescaling method \cite{bernetti2020pressure}, which is an overdamped Langevin equation for $V$ without introducing an artificial mass $M$.
However, the complete limit equations of motion are unclear, particularly for $\b{r}$ and $\b{p}$. 
In fact, the limits for the compressibility term $\dot{V}/(3V)dt=dV/(3V)$ in the equations for $\b{r}$ and $\b{p}$ are unclear.
In the $M>0$ case, $V$ is absolutely continuous and the quadratic variation of $V$ vanishes. However, the quadratic variation of $V$ is nontrivial in the zero-mass limit, which could possibly bring nontrivial contribution for the limit of $dV/(3V)$. 
In fact, one can verify directly that if one use the same equations of $\b{r}$ and $\b{p}$ with the limit equation of $V$, the NPT ensemble is not the invariant measure.

In this paper, we first identify the limit of the equations for $\b{r}$ and $\b{p}$ so that the correct target NPT ensemble can be kept. 
As we shall see, this limit equation is helpful for the justification of some numerical methods used in \cite{bernetti2020pressure}. Furthermore, considering that the volume $V$ should always be positive, we will design the friction coefficients suitably so that this can be guaranteed. For the purpose of discretization for simulations, we employ a change of variable, which together with the suitably chosen friction coefficient would give an additive noise in the equation. Based on this, we design a second order weak scheme using the operator splitting technique. 

The rest of this paper is organized as follows: In Section \ref{sec:EOM}, we derive the zero-mass limit of Eq. \eqref{eq:liang313}, which preserves the target NPT ensemble. Section \ref{sec:first_order} introduces two simple first-order numerical schemes and Section \ref{sec:second-order} introduces a second-order Langevin sampler that preserves the positive volume of the simulation box. Finally, we perform some numerical experiments to validate the effectiveness and applicability of our Langevin sampler across different scenarios in Section \ref{sec:experiments}.

\section{Equations of motion}\label{sec:EOM}

In this section, we first derive the zero-mass limit of Eq. \eqref{eq:liang313}, and in particular, the equations for $\b{r}$ and $\b{p}$. 
Then, we verify that under a specially designed friction coefficient $\gamma(V)$, our equations of motion are well-defined and preserve the correct NPT ensemble.

As mentioned in the introduction, the issue is that the limit of $\b{r}dV/(3V)$ and $\b{p} dV/(3V)$ are unclear. Our key observation is that we may go back to the original system of equations \eqref{eq:originalunderdampedLang}. This system of equations is not used in simulations eventually but it reflects the structure of the system. Note that the Jacobian of this change of variable for $(\b{s}, \b{p}^{\b{s}})\leftrightarrow (\b{r}, \b{p})$ is 1. One may find that the equations of motion are then given by: 
\begin{gather}\label{eq:spsVpV}
    \begin{split}
        & d\b{s} = V^{-2/3}\b{m}^{-1}\b{p}^{\b{s}} dt, \\
        & d\b{p}^{s} = -V^{1/3}\nabla_{\b{r}} U dt - \b{\gamma} \b{p}^{\b{s}} dt + V^{1/3}\sqrt{2\beta^{-1}\b{\gamma}\b{m}}\, d\b{W},\\
        &dV = \frac{p^V}{M}dt, \, dp^V = -\frac{\partial H}{\partial V}\,dt - \tilde{\gamma}(V)p^Vdt + \sqrt{2\beta^{-1}\tilde{\gamma}(V)M}\,dW_V.
    \end{split}
\end{gather}
Using the new variables, it is clear that the equations for $\b{s}$ and $\b{p}^{\b{s}}$ are much more straightforward and there are no trouble terms similar to  $\b{r}dV/(3V)$ or $\b{p} dV/(3V)$.
This gives possibility to derive the correct limit equations of motion for the NPT ensemble.

Consider the rescaled friction
\begin{gather}\label{eq:gamma}
    \gamma(V)=M\tilde{\gamma}(V),
\end{gather}
which we shall fix in the zero mass limit $M\to 0$. According to the Smoluchowski–Kramers approximation result in \cite[Theorem 1]{hottovy2015smoluchowski}, if the functions involved are sufficiently nice,  the $V$-component of the solution of Eq. \eqref{eq:spsVpV}, denoted by $V^M_t$ to indicate the dependence on $M$, converges to the solution of 
\begin{gather}\label{eq:V_general}
    dV = -\frac{1}{\gamma(V)}\left(\frac{\partial\cH}{\partial V} + \frac{1}{\beta\gamma(V)}\frac{d\gamma(V)}{dV}\right)dt + \sqrt{\frac{2}{\beta\gamma(V)}}dW_V,
\end{gather}
where 
\begin{gather}
\cH(\b{s},\b{p}^{\b{s}},V) = U(V^{1/3}\b{s}; V)  + \frac{1}{2}V^{-2/3}(\b{p}^{\b{s}})^T\b{m}^{-1}\b{p}^{\b{s}}+P_0V.
\end{gather}
Here, $\partial \cH/\partial V$ is taken by fixing $\b{s}$ so that $\cH$ also satisfies
\begin{gather}
\frac{\partial \cH}{\partial V}|_{\b{s}}=P_0-P,
\end{gather}
where $P$ is the instantaneous pressure introduced in \eqref{eq:inspressure1}.
Note that
\[
U(V^{1/3}\b{s}; V) + \frac{1}{2}V^{-2/3}(\b{p}^{s})^T\b{m}^{-1}\b{p}^{\b{s}}=U(\b{r}; V) + \frac{1}{2}\b{p}^T\b{m}^{-1}\b{p} ,
\]
so the physical meaning of $\cH$ is actually the enthalpy of the system.
Denoting the solution of Eq. \eqref{eq:V_general} by $V_t$, the Smoluchowski–Kramers approximation result in  \cite[Theorem 1]{hottovy2015smoluchowski} states that, under certain assumptions (the function $\cH$ is sufficiently nice) and the same initial condition $V^M_0 = V_0$, it holds that
\begin{gather}\label{eq:limitofV}
    \lim_{M\to 0} \mathbb{E}\left[\left(\sup_{0\leq t\leq T}\left|V^{M}_t-V_t\right|\right)^2\right] = 0.
\end{gather}
We remark that our system of equations look more complicated as there is coupling with the equations of $\b{r}$
and $\b{p}$ so the rigorous justification for our system is still unavailable. Nevertheless, we will assume \eqref{eq:limitofV} to investigate the limit of the equations for the first two variables.

By \eqref{eq:limitofV}, there is a sequence that converges almost surely under the uniform norm. We will then fix this sequence of $V_t^M$ without relabelling.  Below, we consider a simplified problem to derive the limit equations. That is, we fix a sequence of convergent $V_t^M$ under the uniform convergence norm, and consider the limit. In other words, the sequence of $V_t^M$ is thus given, and we aim to show that the $(\b{s}, \b{p}^{\b{s}})$-component of the solution of Eq. \eqref{eq:spsVpV}, denoted by $(\b{s}^M_t, \b{p}^{\b{s},M}_t)$, converges in $L^2$, with respect to the topology on $C_{\R^{6N}}([0,T])$, to the $(\b{s}, \b{p}^{\b{s}})$-component of the solution of the following equations
\begin{gather}\label{eq:spsV}
    \begin{split}
         d\b{s} &= \frac{\partial \cH}{\partial \b{p}^s}\,dt=V^{-2/3}\b{m}^{-1}\b{p}^{\b{s}} dt, \\
         d\b{p}^{\b{s}} &=-\frac{\partial \cH}{\partial \b{s}}\,dt- V^{2/3}\b{\gamma m} \frac{\partial H}{\partial\b{p}^{\b{s}}} dt + V^{1/3}\sqrt{2\beta^{-1}\b{\gamma}\b{m}}\, d\b{W} \\
         &=-V^{1/3}\nabla_{\b{r}} U dt - \b{\gamma} \b{p}^{\b{s}} dt + V^{1/3}\sqrt{2\beta^{-1}\b{\gamma}\b{m}}\, d\b{W},\\
        dV &= -\frac{1}{\gamma(V)}\left(\frac{\partial\cH}{\partial V} + \frac{1}{\beta\gamma(V)}\frac{d\gamma(V)}{dV}\right)dt + \sqrt{\frac{2}{\beta\gamma(V)}}dW_V,
    \end{split}
\end{gather} 
if they share the same initial conditions.

\begin{proposition}\label{prop:convergence}
    Assume that the potential $U$ is twice continuously differentiable and Eq. \eqref{eq:spsVpV} and Eq. \eqref{eq:spsV} share the same initial conditions. Let the probability space for $W_V$ be $(\Omega_1,\mathcal{F}_1 ,\P_1)$ and the probability space for $\b{W}$ be $(\Omega_2, \mathcal{F}_2, \P_2)$. We assume also that for $\P_1$-almost every $\omega_1\in \Omega_1$, the $V$-component of the solution of Eq. \eqref{eq:spsV} has positive upper and lower bounds (the bounds could depend on $\omega_1$), and the convergence of $V_t^M$ with respect to $\omega_2\in \Omega_2$ is uniform. Then, it holds that
    \begin{gather}\label{eq:sps_conv}
        \lim_{k\to\infty}\mathbb{E}\left[\sup_{0\leq t\leq T}\left|\b{s}^{M}_t-\b{s}_t\right|^2 + \sup_{0\leq t\leq T}\big|\b{p}^{\b{s},M}_t-\b{p}^{\b{s}}_t\big|^2 \Big| \mathcal{F}_1 \right] = 0,
    \end{gather}
    where $(\b{s}_t, \b{p}^{\b{s}}_t)$ is the $(\b{s}, \b{p}^{\b{s}})$-component of the solution of Eq. \eqref{eq:spsV}.
\end{proposition}

\begin{proof}[Proof of Proposition \ref{prop:convergence}]
  
For the conditioning expectation, it is enough to fix $\omega_1\in \Omega_1$ such that $V_t^M$ converges to $V_t$ (since such $\omega_1$ has probability $1$). By the assumption, we can find $V_{u}'>V_l'>0$ such that $V_l' \le V_t\le V_u'$ for this fixed $\omega_1$. Moreover, by the assumption on the convergence of $V_t^M$, we find that for $M$ small enough, one has
$V_l:=V_l'/2\le V_t^M \le 2V_u'=:V_u$.
  Below in the proof here, we will use $\E(\cdot)$ to indicate the expectation on $\Omega_2$ by fixing such $\omega_1\in \Omega_1$, for the ease of notations.
  
    Note that $U$ is twice continuously differentiable and thus there exists $C_{0}>0$ such that 
    \begin{gather*}
        |\nabla_{\b{r}}U|, |\nabla_{\b{r}}\partial_V U|, |\nabla_{\b{r}}^2U| \leq C_{0},\, \forall V\in[V_l, V_u], \forall \b{s}\in\T_1^d.
    \end{gather*}
    where all these functions are evaluated at $(V^{1/3}\b{s}; V)$. With this, it is straightforward to find that
  \begin{gather*}
  \E \sup_{0\le t'\le t}|\b{p}^{\b{s}}_{t'}|^2\le C+C\E\int_0^t\sup_{0\le s'\le s}|\b{p}^{\b{s}}_{s'}|^2ds
  +C(T)\E\sup_{0\le t'\le t}\left|\int_0^{t'}V_s^{1/3}d\b{W}_s\right|^2
  \end{gather*}
  Applying the Burkholder-Davis-Gundy (BDG) inequality \cite{barlow1982semi, carlen1991lp}, the last term is bounded. Hence,
 $\E \sup_{0\le t'\le t}|\b{p}^{\b{s}}_{t'}|^2\le C(T)$ by Gr\"onwall's inequality. 

Due to the boundedness of these quantities, it is straightforward to find
 \begin{gather*}
        \begin{split}
            |\b{s}^{M}_t-\b{s}_t| &\le  C\int_0^t |\b{p}^{\b{s},M}_{s} - \b{p}^{\b{s}}_{s}|ds +|V^{M}_{s}-V_{s}| |\b{p}^{\b{s}}_{s}| \, ds,\\ 
            |\b{p}^{\b{s},M}_t-\b{p}^{\b{s}}_t| &\le  C\int_0^t |V^{M}_{s}-V_{s}|\,ds +C \int_0^1 |\b{s}^{M}_{s}-\b{s}_{s}|+|\b{p}^{\b{s},M}_{s}-\b{p}^{\b{s}}_{s}|\, ds\\
            & \quad\quad+ \left|\int_0^t\left[(V_s^{M})^{1/3}-V_t^{1/3}\right]\sqrt{2\beta^{-1}\b{\gamma m}}\,d\b{W}_{s}\right|.
        \end{split}
 \end{gather*}
Taking the square on both sides, taking the supremum over $t\in [0, T]$, one has
 \begin{gather*}
\begin{split}
\E\sup_{0\le t'\le t} |\b{s}^{M}_{t'}-\b{s}_{t'}|^2 &\le  C(T)\int_0^t \E \sup_{0\le s'\le s}|\b{p}^{\b{s},M}_{s'} - \b{p}^{\b{s}}_{s'}|^2ds \\
&\quad+C(T)\int_0^T\Big\|\sup_{0\le t_1\le T}|V^{M}_{t_1}-V_{t_1}|\Big\|_{\infty}^2 \E \sup_{0\le s'\le s} |\b{p}^{\b{s}}_{s'}|^2 \, ds,\\ 
 \E\sup_{0\le t'\le t} |\b{p}^{\b{s},M}_{t'}-\b{p}^{\b{s}}_{t'}|^2 &\le  C(T)\Big\|\sup_{0\le t_1\le T}|V^{M}_{t_1}-V_{t_1}|\Big\|_{\infty}^2 +C(T) \int_0^t 
 \E \sup_{0\le s'\le s} |\b{s}^{M}_{s'}-\b{s}_{s'}|^2\,ds\\
&\quad + C(T)\int_0^t\E \sup_{0\le s'\le s} |\b{p}^{\b{s},M}_{s'}-\b{p}^{\b{s}}_{s'}|^2\, ds+R_1(t),
        \end{split}
 \end{gather*}
 where
 \[
 R_1(t)=C(T) \E \sup_{0\le t'\le t}\left|\int_0^{t'}\left[(V_s^{M})^{1/3}-V_s^{1/3}\right]\sqrt{2\beta^{-1}\b{\gamma m}}\,d\b{W}_{s}\right|^2.
 \]
Applying the BDG inequality again, one finds that
\[
R_1(t)\le C(T)\E \sup_{0\le t'\le t} \int_0^{t'}|(V_s^{M})^{1/3}-V_s^{1/3}|^2\,ds\le C(T)\Big\|\sup_{0\le t_1\le T}|V^{M}_{t_1}-V_{t_1}|\Big\|_{\infty}^2.
\]
  According to the Gr\"onwall's inequality, it holds that
    \begin{gather*}
        \mathbb{E}\left[\sup_{0\leq t'\leq t}\big\|\b{s}^{M}_{t'}-\b{s}_{t'}\big\|^2 + \sup_{0\leq t'\leq t}\big\|\b{p}^{\b{s},M}_{t'}-\b{p}^{\b{s}}_{t'}\big\|^2\right] \leq C(T)\Big\|\sup_{0\le t_1\le T}|V^{M}_{t_1}-V_{t_1}|\Big\|_{\infty}^2.
    \end{gather*}
Hence, the claim holds.
\end{proof}

Note that the assumptions on the convergence of $V_t^M$ in Proposition \ref{prop:convergence} seems a little bit strong. However,
since the Smoluchowski–Kramers approximation result in  \cite[Theorem 1]{hottovy2015smoluchowski} is about the randomness brought by $W_V$, while the randomness in $\b{W}$ only affects $V$ through $\cH$. As we have assumed that $\cH$ is nice, we can expect that the assumption is reasonable. One may investigate the convergence here rigorously in the future (the problem for the case without $\b{W}$ is also interesting). 
Though this justification cannot be used as a rigorous proof, we regard it as a convincing derivation. Besides, we will start with the limit equations Eq. \eqref{eq:spsV} directly to show that it has the desired properties and it can preserve the correct NPT ensemble.

From Eq. \eqref{eq:spsV}, one can change back to the variables $\b{r}=V^{1/3}\b{s}$ and $\b{p}=V^{-1/3}\b{p}^{\b{s}}$ by applying It\^o's formula. Then, one can find that the equations for $\b{r}$ and $\b{p}$ are given by the following.
\begin{gather}\label{eq:rpV_general}
    \begin{split}
        & d\b{r} = \b{m}^{-1}\b{p}\,dt + \frac{dV}{3V}\b{r} - \frac{2}{9\beta V^2\gamma(V)}\b{r}\,dt,\\
        & d\b{p} = -\nabla_{\b{r}} U\,dt - \frac{dV}{3V}\b{p} + \frac{4}{9\beta V^2\gamma(V)}\b{p}\,dt - \b{\gamma} \b{p}\,dt + \sqrt{2\beta^{-1}\b{m}\b{\gamma}} \,d\b{W}.
    \end{split}
\end{gather}
Clearly, some extra drift terms $- \frac{2}{9\beta V^2\gamma(V)}\b{r}\,dt$ and $\frac{4}{9\beta V^2\gamma(V)}\b{p}\,dt$ originate from the quadratic variation of $V$. Expressed in terms of the side length $L=V^{1/3}$ of the cubic periodic box, the equations look more concise under the following identities:
\begin{gather*}
\frac{dL}{L} = \frac{dV}{3V} - \frac{2}{9\beta V^2\gamma(V)} \quad\text{and}\quad \frac{dL^{-1}}{L^{-1}} = - \frac{dV}{3V} + \frac{4}{9\beta V^2\gamma(V)}.
\end{gather*}
To ensure that the volume is always positive during the evolution, we apply another change of variable $\epsilon=\log V$ similar as in \cite{bernetti2020pressure}. Then, the equations of motion in our work are given by the following system of equations
\begin{gather}\label{eq:rpLV_general}
    \begin{split}
        & d\b{r} = \b{m}^{-1}\b{p}\,dt + \b{r}\frac{dL}{L},\\
        & d\b{p} = -\nabla_{\b{r}} U\,dt + \b{p}\frac{dL^{-1}}{L^{-1}} - \b{\gamma} \b{p}\,dt + \sqrt{2\beta^{-1}\b{m}\b{\gamma}} dW,\\
        & L=V^{1/3}, V = e^{\epsilon},\\
        &d\epsilon = -\frac{1}{V\gamma(V)}\left(\frac{\partial\cH}{\partial V} + \frac{1}{\beta \gamma(V)} \frac{d\gamma(V)}{d V} + \frac{1}{\beta V}\right)dt + \sqrt{\frac{2}{\beta V^2\gamma(V)}}dW_V.
    \end{split}
\end{gather}
We remark that taking $\gamma(V)=\tau_p/(\beta_TV)$ in Eq. \eqref{eq:rpLV_general}, one can recover the recently proposed stochastic cell rescaling method \cite{bernetti2020pressure}, where the evolution of $\epsilon$ is described by the following overdamped Langevin equation:
\begin{gather}\label{eq:bernettieq5}
    d\epsilon = -\frac{\beta_T}{\tau_p}\frac{\partial\cH}{\partial V} dt + \sqrt{\frac{2\beta_T}{\beta\tau_p V}}dW_V.
\end{gather}
Besides the equation for $V$, our equations of motion in \eqref{eq:rpLV_general} give the full complete dynamics.

In this work, we aim to propose an equation that guarantees the positivity of $V$ (that is, the equation for $d\epsilon$ is well-defined), and moreover an equation that is easy for discretization if we desire a second order scheme.
To this end, we choose the friction coefficient to be 
\begin{gather}\label{eq:friction}
\gamma(V) = 1/(\lambda V^2),
\end{gather}
 with $\lambda>0$ such that the equation for $\epsilon=\log V$ exhibits a constant diffusion coefficient (i.e., additive noise).
 \begin{gather}\label{eq:eps}
    d\epsilon = -\lambda\left(V\frac{\partial\cH}{\partial V} - \beta^{-1}\right)dt + \sqrt{2\lambda\beta^{-1}} dW_V.
\end{gather}
With the additive noise,  some relatively simple second-order numerical schemes can be employed for discretization, which will be discussed in Section \ref{sec:second-order}.  Moreover, with additive noise, the positivity of $V$ can be shown by assuming that $U$ is sufficiently smooth.

In the following, we verify that \eqref{eq:rpLV_general} with the choice \eqref{eq:friction}  is well-defined and our the system of equations preserves the target NPT ensemble. We start with the following assumption.
\begin{assumption}\label{ass:V2partial2H}
    For any $\delta>0$, there exist $V_{\delta}>0$ and $C_{\delta}>0$, such that for any $V\ge V_{\delta}$, it holds that
    \begin{gather}\label{eq:V2partial2H}
        V^2\left|\frac{\partial^2\mathcal{H}}{\partial V^2} \Big|_{\b{s}}\right| \leq \delta V^2 \left(\frac{\partial\mathcal{H}}{\partial V}\Big|_{\b{s}}\right)^2 + C_{\delta}(\mathcal{H}+1).
    \end{gather}
    Here, the symbol $\frac{\partial}{\partial V}\Big|_{\b{s}}$ means the partial derivative is taken by holding $\b{s}=V^{-1/3}\b{r}$ fixed.
\end{assumption}

The Assumption \ref{ass:V2partial2H} holds in our setting if $U$ is nice enough. In fact, for the periodic case, the potential $U$ is given by
\begin{gather}\label{eq:potentialU}
U(\b{r}; V) = \sum_{i<j}\sum_{\b{n}}\Phi(\b{r}_{ij} + V^{1/3}\b{n}),
\end{gather}
 where $\b{r}_{ij} = \b{r}_i - \b{r}_j$ and $\b{n}\in\mathbb{Z}^3$ ranges over the three-dimensional integer column vectors.
If the $\Phi$ decays fast enough, we can establish the following result.
\begin{lemma}
Consider the potential function in \eqref{eq:potentialU}. Suppose that $\Phi: \R^3\rightarrow\R$ is a twice continuously differentiable radially symmetric potential with the derivatives $|D^{\alpha}\Phi(\b{r})|$ decay fast enough as $|\b{r}|\to\infty$ for $|\alpha|\le 2$. Then, Assumption \ref{ass:V2partial2H} holds.
\end{lemma}

\begin{proof}
A direct computation shows that 
\begin{gather}\label{eq:VpartialH}
V\frac{\partial\cH}{\partial V} = P_0V + \frac{1}{3} \sum_{i<j}\sum_{\b{n}}\nabla\Phi(\b{r}_{ij} + V^{1/3}\b{n})^T(\b{r}_{ij} + V^{1/3}\b{n}) - \frac{4}{3}K
\end{gather}
 and 
 \begin{multline*}
 V^2\frac{\partial^2\cH}{\partial V^2} = \frac{1}{9}\sum_{i<j}\sum_{\b{n}}(\b{r}_{ij} + V^{1/3}\b{n})^T\nabla^2\Phi(\b{r}_{ij} + V^{1/3}\b{n})(\b{r}_{ij} + V^{1/3}\b{n}) \\
 - \frac{2}{9}\sum_{i<j}\sum_{\b{n}}\nabla\Phi(\b{r}_{ij} + V^{1/3}\b{n})^T(\b{r}_{ij} + V^{1/3}\b{n}) + \frac{20}{9}K,
 \end{multline*}
  where $K = \frac{1}{2}\b{p}^T\b{m}^{-1}\b{p}$ is the kinetic energy of the system.

With the assumptions, both the series $\sum_{i<j}\sum_{\b{n}}\nabla\Phi(\b{r}_{ij} + V^{1/3}\b{n})^T(\b{r}_{ij} + V^{1/3}\b{n})$ and $\sum_{i<j}\sum_{\b{n}}(\b{r}_{ij} + V^{1/3}\b{n})^T\nabla^2\Phi(\b{r}_{ij} + V^{1/3}\b{n})(\b{r}_{ij} + V^{1/3}\b{n})$ would be bounded uniformly with respect to $V\ge V_c$ for some fixed $V_c>0$. 

Thus, for any large $V$,  the term $K$ would be bounded by $\cH$ while the other terms would be dominated by $P_0V$. 
\end{proof}

\begin{theorem}\label{thm:invariant_measure}
    Set the friction coefficient $\gamma(V)=1/(\lambda V^2)$ and consider a given initial value $(\b{r}_0, \b{p}_0, \epsilon_0)$ so that $V_0=e^{\epsilon_0}>0$. Suppose that the potential $U$ is twice continuously differentiable and Assumption \ref{ass:V2partial2H} holds. Then, Eq. \eqref{eq:rpLV_general} is well-defined and preserves a positive $V$ throughout the evolution. Furthermore, Eq. \eqref{eq:rpLV_general} has the distribution \eqref{eq:invariant_measure} as an invariant measure (in terms of the variables $(\b{r}, \b{p}, V)$). 
\end{theorem}

\begin{proof}

{\bf Step 1. Well-posedness and positivity of $V$}

 For any positive integers $m$ and $n$, we define the stopping times 
 \[
 \tau^u_n = \inf\{t>0: V_t \geq n \text{ and } \|\b{p}^{\b{s}}_t\|^2 \geq n\}
 \]
  and 
  \[
  \tau^l_m = \inf\{t>0: V_t \leq 1/m\}.
  \]
  Then $\tau_{m,n} := \tau^l_m\wedge\tau^u_n$ is also a stopping time.

  We first fix $m$ and verify that with probability one $V_t$ cannot diverge to $+\infty$ before the time $T$. 
    Under $\gamma(V)=1/(\lambda V^2)$, Eq. \eqref{eq:V_general} writes
    \begin{gather}\label{eq:dV_mine}
        dV = -\lambda V^2\frac{\partial\mathcal{H}}{\partial V}\,dt + 2\lambda \beta^{-1} Vdt + \sqrt{2\lambda\beta^{-1}}VdW_V.
    \end{gather}
  
 Here, we will use the variables $(\b{s}_t, \b{p}^s_t, V_t)$ as they are more suited to the underlying dissipation structure.  
 Applying It\^o's formula to $\cH(\b{s}_t, \b{p}^{\b{s}}_t, V_t)$, one has
    \begin{gather}\label{eq:dH}
    \begin{split}
        \int_0^{\tau_{m,n}\wedge t} d\cH_s =& \int_0^{\tau_{m,n}\wedge t} b(\b{s}_s, \b{p}^{\b{s}}_s, V_s)\,ds + \int_0^{\tau_{m,n}\wedge t} V^{\frac{1}{3}}_s\left(\frac{\partial\cH}{\partial \b{p}^{\b{s}}}\right)_s\cdot \sqrt{2\beta^{-1}\b{\gamma m}}d\b{W}_s\\
        &+ \int_0^{\tau_{m,n}\wedge t} \sqrt{2\lambda\beta^{-1}}V_s\left(\frac{\partial\mathcal{H}}{\partial V}\right)_s dW_{V,s},        
    \end{split}
    \end{gather}
    where 
    \begin{multline}\label{eq:drift_H}
        b(\b{s}, \b{p}^{\b{s}}, V) = -\lambda V^2\left[\left(\frac{\partial\mathcal{H}}{\partial V}\right)^2 - \beta^{-1}\frac{\partial^2\mathcal{H}}{\partial V^2}\right] + 2\lambda \beta^{-1}V\frac{\partial\mathcal{H}}{\partial V} \\
        - V^{\frac{2}{3}}
        \frac{\partial \cH}{\partial \b{p}^{\b{s}}}\cdot \b{\gamma m}\frac{\partial \cH}{\partial \b{p}^{\b{s}}}
          + \frac{3}{\beta}\sum_{i=1}^N\gamma_i,
    \end{multline}
 where $\frac{3}{\beta}\sum_{i=1}^N\gamma_i$ is $\beta^{-1}V^{2/3}\frac{\partial^2\cH}{(\partial \b{p}^{\b{s}})^2}:\b{\gamma m}$.    
    
Fix the $\delta$ in Assumption \ref{ass:V2partial2H} to be $\delta=\beta/3$.
Since $V_{s}\ge 1/m$ for $s\le \tau_{m,n}\wedge t$, then $b$ is bounded when $V\le V_{\delta}$. When $V\ge V_{\delta}$, one finds that
\begin{gather*}
\begin{split}
b &\le  - \frac{2}{3}\lambda V^2\left(\frac{\partial\mathcal{H}}{\partial V}\right)^2  + 2\lambda \beta^{-1}V\frac{\partial\mathcal{H}}{\partial V}+C(\cH+1)- V^{\frac{2}{3}} \frac{\partial \cH}{\partial \b{p}^{\b{s}}}\cdot \b{\gamma m}\frac{\partial \cH}{\partial \b{p}^{\b{s}}}
          + \frac{3}{\beta}\sum_{i=1}^N\gamma_i\\
   & \le - \frac{1}{3}\lambda V^2\left(\frac{\partial\mathcal{H}}{\partial V}\right)^2+C'(\cH+1)- V^{\frac{2}{3}} \frac{\partial \cH}{\partial \b{p}^{\b{s}}}\cdot \b{\gamma m}\frac{\partial \cH}{\partial \b{p}^{\b{s}}},
\end{split}
\end{gather*}
where the constant $C'$ is independent of $n$ (may depend on $m$).

Consequently, one has
    \begin{multline*}
\E \cH_{\tau_{m,n}\wedge t} + \E\int_0^{\tau_{m,n}\wedge t} \frac{\lambda}{3} V_s^2 \left(\frac{\partial\mathcal{H}}{\partial V}\right)_s^2 + V^{2/3} \frac{\partial \cH}{\partial \b{p}^{\b{s}}}\cdot \b{\gamma m}\frac{\partial \cH}{\partial \b{p}^{\b{s}}}\,ds\\
            \leq  C(m,T) + C(m,T)\int_0^{t} \E\cH_{\tau_{m,n}\wedge s}\,ds.
    \end{multline*}
    Applying the Gr\"onwall's inequality, one can conclude $\E \cH_{\tau_{m,n}\wedge t}\le C(m,T)$ for $t\le T$, independent of $n$,
    and consequently
    \begin{gather}\label{eq:BDGbound}
        \E\int_{0}^{\tau_n\wedge T} \frac{\lambda}{3} V_t^2 \left(\frac{\partial\mathcal{H}}{\partial V}\right)_t^2 + V^{2/3} \frac{\partial \cH}{\partial \b{p}^{\b{s}}}\cdot \b{\gamma m}\frac{\partial \cH}{\partial \b{p}^{\b{s}}}\,dt \leq C(m,T),
    \end{gather}
    with the bound independent of $n$.
    
 Applying the BDG inequality and Eq. \eqref{eq:BDGbound}, there exist $C^{\prime\prime}>0$, independent of $n$, such that
    \begin{gather}
        \begin{split}
            &\E\sup_{0\leq t\leq T}\left|\int_{0}^{\tau_{m,n}\wedge t} V^{-\frac{1}{3}}_s\b{p}^{\b{s},T}_s\sqrt{\b{\gamma}\b{m}^{-1}}d\b{W}_s\right| + \E\sup_{0\leq t\leq T}\left|\int_{0}^{\tau_{m,n}\wedge t} V_s\left(\frac{\partial\mathcal{H}}{\partial V}\right)_s dW_{V,s}\right| \leq C^{\prime\prime}.\\
        \end{split}
    \end{gather}

Taking the supremum with respect to $t\in[0, T]$ on both sides of Eq. \eqref{eq:dH}, one then has
    \begin{gather}
        \E\sup_{0\leq t\leq T}\cH_{\tau_{m,n}\wedge t} \leq C(m,T) + C(m,T)\int_{0}^{T} \E\sup_{0\leq s\leq t}\cH_{\tau_{m,n}\wedge s}\,dt.
    \end{gather}
 This means that $\E\sup_{0\leq t\leq T}\cH_{\tau_{m,n}\wedge t}\le C(m, T)$ independent of $n$ by Gr\"onwall's inequality. 
Since the potential is bounded, one thus concludes by Markov's inequality that
    \begin{gather}
        \lim_{n\to\infty}\mathbb{P}(\tau_{n}^u \leq T\wedge \tau_m^l) = 0.
    \end{gather}
 The event $\{\tau_{n}^u \leq T\wedge \tau_m^l \}$ is decreasing as $n\to\infty$, so $V$ and $|\b{p}^{\b{s}}|$ will not go to $+\infty$ in finite time with probability one.

For the stopping time $\tau_m^l$, we need to verify that with probability one, $V_t$ cannot tend to 0 in the time interval $[0, T]$. 
For this purpose, we consider Eq. \eqref{eq:eps}, or
    \begin{multline}\label{eq:deps_sps}
        d\epsilon = \frac{\lambda}{3}\left[V^{-\frac{2}{3}}\b{p}^{\b{s},T}\b{m}^{-1}\b{p}^{\b{s}} - V^{\frac{1}{3}}\b{s}^T\nabla_{\b{r}}U - 3V\left(\frac{\partial U}{\partial V}\Big|_{\b{s}} + P_0\right)\right]dt \\
        + \lambda\beta^{-1}dt + \sqrt{2\lambda\beta^{-1}} dW_V.
    \end{multline}
When $\epsilon<0$ or $V<1$, 
\[
\frac{\lambda}{3}\left[V^{-\frac{2}{3}}\b{p}^{\b{s},T}\b{m}^{-1}\b{p}^{\b{s}} - V^{\frac{1}{3}}\b{s}^T\nabla_{\b{r}}U - 3V\left(\frac{\partial U}{\partial V}\Big|_{\b{s}} + P_0\right)\right]\ge -C_1,
\]
which means that the drift is bounded below. Since the noise $\sqrt{2\lambda\beta^{-1}} dW_V$ is additive, it is easy to show that 
\[
\lim_{m\to\infty}\mathbb{P}\left[\tau^l_m \leq T\right] = 0.
\]
The event $\{\tau^l_m \leq T\}$ is decreasing as $m\to\infty$, and thus $V_t = e^{\epsilon_t}$ cannot tend to 0 in finite time with probability one. 
    
Since the system of equations is well-posed for $(\b{s}, \b{p}^s, \epsilon)$, $V=e^{\epsilon}$ is thus positive for any fixed time interval.

\noindent {\bf Step 2. Invariant measure}

    We now prove the preservation of the distribution \eqref{eq:invariant_measure} under Eq. \eqref{eq:rpLV_general}. 
    Here, we will verify the claim using variables $(\b{s}, \b{p}^s, \epsilon)$. The probability distribution for NPT ensemble under the new variables $(\b{s}, \b{p}^s, \epsilon)$ is rewritten as 
\begin{gather*}
    d\mu \propto \exp\left(-\beta\cH(e^{\epsilon/3}\b{s},e^{-\epsilon/3}\b{p}^{\b{s}},e^{\epsilon})\right)e^{\epsilon}d\b{s}d\b{p}^{\b{s}}d\epsilon,
\end{gather*}
where the extra $e^{\epsilon}=\frac{dV}{d\epsilon}$ originates from the change of variable $V\to e^{\epsilon}$. 
The Fokker-Planck equation for the variables $(\b{s}, \b{p}^s, \epsilon)$ (the first two equations in \eqref{eq:spsV} and \eqref{eq:eps}) is given by
\begin{multline}
\partial_t\rho=
        - \nabla_{\b{s}}\cdot\left(e^{-\frac{2\epsilon}{3}}\b{m}^{-1}\b{p}^{\b{s}}\rho \right) - \nabla_{\b{p}^{\b{s}}}\cdot\left(\left(-e^{\frac{\epsilon}{3}}\nabla_{\b{r}} U - \b{\gamma} \b{p}^{\b{s}}\right)\rho \right)\\
        - \partial_{\epsilon}\left[-\lambda \left(V\frac{\partial\cH}{\partial V} - \beta^{-1}\right)\rho \right]
        + \beta^{-1}\nabla^2_{\b{p}^{\b{s}}}:\left(\b{m}\b{\gamma}\rho \right) + \lambda \beta^{-1}\partial_{\epsilon}^2\rho.
\end{multline}

One can easily compute that
\begin{gather*}
    \begin{split}
        & -\nabla_{\b{s}}\cdot\left(e^{-\frac{2\epsilon}{3}}\b{m}^{-1}\b{p}^{\b{s}}e^{-\beta\cH}e^{\epsilon}\right) - \nabla_{\b{p}^{\b{s}}}\cdot\left(\left(-e^{\frac{\epsilon}{3}}\nabla_{\b{r}} U - \b{\gamma} \b{p}^{\b{s}}\right)e^{-\beta\cH}e^{\epsilon}\right)\\
        =&e^{-\beta\cH}e^{\epsilon}\left(3\sum_i \gamma_i -\beta e^{-\frac{2\epsilon}{3}}\b{p}^{\b{s},T}\b{\gamma} \b{m}^{-1}\b{p}^{\b{s}}\right),\\
        & \beta^{-1}\nabla^2_{\b{p}^{\b{s}}}:\left(\b{m}\b{\gamma}e^{-\beta\cH}e^{\epsilon}\right) = e^{-\beta\cH}e^{\epsilon}\left(-3\sum_i \gamma_i + \beta e^{-\frac{2\epsilon}{3}}\b{p}^{\b{s},T}\b{\gamma} \b{m}^{-1}\b{p}^{\b{s}}\right)\\
    \end{split}
\end{gather*}
Hence, 
\[
-\nabla_{\b{s}}\cdot\left(e^{-\frac{2\epsilon}{3}}\b{m}^{-1}\b{p}^{\b{s}}e^{-\beta\cH}e^{\epsilon}\right) - \nabla_{\b{p}^{\b{s}}}\cdot\left(\left(-e^{\frac{\epsilon}{3}}\nabla_{\b{r}} U - \b{\gamma} \b{p}^{\b{s}}\right)e^{-\beta\cH}e^{\epsilon}\right)
+\beta^{-1}\nabla^2_{\b{p}^{\b{s}}}:\left(\b{m}\b{\gamma}e^{-\beta\cH}e^{\epsilon}\right)=0.
\]
Moreover
\begin{gather*}
    \lambda\beta^{-1}\partial_{\epsilon}\left(e^{-\beta\cH}e^{\epsilon}\right) = -\lambda \left(e^{\epsilon}\frac{\partial\cH}{\partial V} - \beta^{-1}\right)e^{-\beta\cH}e^{\epsilon}.
\end{gather*}
Hence,
\[
- \partial_{\epsilon}\left[-\lambda \left(V\frac{\partial\cH}{\partial V} - \beta^{-1}\right)e^{-\beta\cH}e^{\epsilon} \right]+ \lambda \beta^{-1}\partial_{\epsilon}^2(e^{-\beta\cH}e^{\epsilon})=0.
\]
Therefore, these terms add to zero so that $e^{-\beta\cH}e^{\epsilon}$ is indeed a stationary solution of the Fokker-Planck equation, which indicates that \eqref{eq:invariant_measure} is the invariant measure of Eq. \eqref{eq:rpLV_general}. 
\end{proof}

Note that the equations for $\b{r}$ and $\b{p}$ do not appear in \cite{bernetti2020pressure}, and the schemes for $\b{r}$ and $\b{p}$ follow from a simple cell rescaling. In this sense, our system of equations of motion \eqref{eq:rpLV_general} concurrently substantiates the validity of the cell rescaling method. This method, though works OK in practice, cannot guarantee that the volume is positive.

\section{Some first-order schemes for the NPT ensemble}\label{sec:first_order}

In this section, we consider some naive discretization of Eq. \eqref{eq:rpLV_general}, which would be some first-order schemes. 
These schemes, though cannot guarantee the positivity of the volume, might be used in practice due to their simplicity.

Firstly, we recall the weak order in numerical SDE. 
\begin{definition}
A method with the one-step propagator $\mathcal{S}^*$ is said to have weak order $p\ge 1$ if for any smooth test function $\varphi\in C_p^{\infty}$ and initial density $\rho_0$ that has all orders of moments, one has
\begin{gather*}
|\langle \varphi, \mathcal{S}^*\rho_0-e^{\Delta t\cL^*}\rho_0\rangle|
=|\langle (\mathcal{S}-e^{\Delta t\cL})\varphi, \rho_0\rangle |\le C_{\varphi}\Delta t^{p+1}.
\end{gather*}
The method is said to converge with weak order $p$ if for all smooth test function $\varphi\in C_p^{\infty}$ and initial density $\rho_0$,  one has
\begin{gather*}
\sup_{n: n\Delta t\le T}|\langle \varphi, (\mathcal{S}^*)^n\rho_0-e^{n\Delta t\cL^*}\rho_0\rangle|
=\sup_{n: n\Delta t\le T}|\langle (\mathcal{S}^n-e^{n\Delta t\cL})\varphi, \rho_0\rangle |\le C_{\varphi}\Delta t^p.
\end{gather*}
\end{definition}

\subsection*{First order scheme \uppercase\expandafter{\romannumeral1}}
Applying the Euler-Maruyama scheme to the equation for $V$ and approximating $dL/L$ and $dL^{-1}/L^{-1}$ by $(L_{n+1}-L_{n})/L_{n}$ and $(L_{n+1}^{-1}-L_{n}^{-1})/L_{n}^{-1}$, respectively, Eq. \eqref{eq:rpLV_general} is then discretized as
\begin{gather}\label{eq:EM_Vrp}
    \begin{split}
        V_{n+1} =& V_{n}-\frac{\Delta t}{\gamma(V_{n})}\left[\left(\frac{\partial\cH}{\partial V}\right)_n + \frac{1}{\beta \gamma(V_{n})}\frac{d\gamma(V_n)}{dV}\right] + \sqrt{\frac{2\Delta t}{\beta \gamma(V_{n})}} \,z^V_n,\\
        \b{r}_{n+1}=&\b{m}^{-1}\b{p_{n}} \Delta t + \b{r_{n}}\frac{L_{n+1}}{L_{n}},\quad  L_{n+1} =V_{n+1}^{1/3}, \\
        \b{p}_{n+1}=&-\left(\nabla_{\b{r}}U\right)_n \Delta t + \b{p_{n}}\frac{L_{n}}{L_{n+1}}-\b{\gamma}\b{p_{n}}\Delta t + \sqrt{2/\beta\Delta t\b{\gamma m}}\,\b{z}_n,\\
    \end{split}
\end{gather}
where $\left(\frac{\partial\cH}{\partial V}\right)_n := \frac{\partial \cH}{\partial V}(\b{r_{n}}, \b{p_{n}}; V_n)$ and $\left(\nabla_{\b{r}}U\right)_n := \nabla_{\b{r}}U(\b{r_{n}}; V_n)$ . 

The Scheme \eqref{eq:EM_Vrp}, although discretized naively using the Euler-Maruyama scheme, coincides with the stochastic cell rescaling method \cite{bernetti2020pressure} in the sense that the effects of the terms $\b{r_{n}}\frac{L_{n+1}}{L_{n}}$ and $\b{p_{n}}\frac{L_{n}}{L_{n+1}}$ in Scheme \eqref{eq:EM_Vrp} correspond to the position-scaling and velocity-scaling steps in the stochastic cell rescaling method. Note that the equations for $\b{r}$ and $\b{p}$ do not appear in \cite{bernetti2020pressure}, and the schemes for $\b{r}$ and $\b{p}$ follow from a simple cell rescaling. In this sense, our system of equations of motion, Eq. \eqref{eq:rpLV_general}, concurrently substantiates the validity of the cell rescaling method. 
By applying the idea of operator splitting to the equations for $\b{r}$ and $\b{p}$, one can recover the Trotter integrator discussed in \cite{bernetti2020pressure}, which is also a first-order scheme. 
\begin{gather}\label{eq:Trotter}
    \begin{split}
        \widetilde{\b{p}_{n+\frac{1}{2}}} &= e^{-\frac{\Delta t}{2}\b{\gamma}}\b{p}_{n} + \sqrt{(1-e^{-\Delta t\b{\gamma}})\beta^{-1}\b{m}}\,\b{z}_n^{(1)},\\
        \b{p}_{n+\frac{1}{2}} &= \widetilde{\b{p}_{n+\frac{1}{2}}}-\frac{\Delta t}{2}\nabla_{\b{r}}U(\b{r}_n),\\
        V_{n+1} &= V_{n}-\frac{\Delta t}{\gamma(V_{n})}\left[\left(\frac{\partial\cH}{\partial V}\right)_n + \frac{1}{\beta \gamma(V_{n})}\frac{d\gamma(V_n)}{dV}\right] + \sqrt{\frac{2\Delta t}{\beta \gamma(V_{n})}} \,z^V_n,\\
        \b{r}_{n+\frac{1}{2}} &= \b{r}_n + \frac{\Delta t}{2}\b{m}^{-1}\b{p}_{n+\frac{1}{2}},\\
        \b{p}_{n+\frac{1}{2}}^{\star} &= \b{p}_{n+\frac{1}{2}}\frac{L_n}{L_{n+1}}, \quad \b{r}_{n+\frac{1}{2}}^{\star} = \b{r}_{n+\frac{1}{2}}\frac{L_{n+1}}{L_n}, \quad L_{n+1} = V_{n+1}^{1/3},\\
        \b{r}_{n+1} &= \b{r}_{n+\frac{1}{2}}^{\star} + \frac{\Delta t}{2}\b{m}^{-1}\b{p}_{n+\frac{1}{2}}^{\star},\\
        \widetilde{\b{p}_{n+1}} &= \b{p}_{n+\frac{1}{2}}^{\star}-\nabla_{\b{r}}U(\b{r}_{n+1})\frac{\Delta t}{2},\\
        \b{p}_{n+1} &= e^{-\frac{\Delta t}{2}\b{\gamma}}\widetilde{\b{p}_{n+1}} + \sqrt{(1-e^{-\Delta t\b{\gamma}})\beta^{-1}\b{m}}\,\b{z}_n^{(2)}.
    \end{split}
\end{gather}

Viewing the dynamics of $(\b{r}, \b{p}, V)$ as a single SDE, one can readily conclude the weak first-order convergence of Scheme \eqref{eq:EM_Vrp} and the Trotter integrator. 
\begin{proposition}\label{prop:weak1}
 Assume that $U$ is a twice continuously differentiable function. The weak order of  both Scheme \eqref{eq:EM_Vrp} and the Trotter integrator (Scheme \eqref{eq:Trotter}) is 1. 
\end{proposition}
This is because the equation for $V$ is directly discretized using the Euler-Maruyama scheme, even though the Trotter integrator evolves the thermostat part of Eq. \eqref{eq:rpLV_general} exactly and discretizes the Hamiltonian system part using the second-order velocity Verlet algorithm. We skip the details of the verification.

\subsection*{First order scheme \uppercase\expandafter{\romannumeral2}}

As we have seen in \eqref{eq:rpV_general}, in the limit equations for $\b{r}$ and $\b{p}$, there are extra drifts arising from the $dV/(3V)$ term. However, it is possible to choose a suitable time point in $V$ such that the scheme looks like the original equation \eqref{eq:liang313}. In particular, we aim to seek $V_{\b{r}}^{\star}$ and $V_{\b{p}}^{\star}$ such that the following numerical scheme is a consistent scheme to the equations of motion. 
\begin{gather}\label{eq:rpV_temp}
    \begin{split}
        &\b{r}_{n+1}=\b{r}_{n}+\b{m}^{-1}\b{p}_{n}\Delta t+\b{r}_{n}\frac{V_{n+1}-V_{n}}{3V_{\b{r}}^{\star}},\\
        &\b{p}_{n+1}=\b{p}_{n}-\nabla_{\b{r}}U\Delta t-\b{p}_{n}\frac{V_{n+1}-V_{n}}{3V_{\b{p}}^{\star}}-\b{\gamma p}_{n}\Delta t+\sqrt{2/\beta\b{m\gamma}\Delta t}\,\b{z}_n,\\
        &V_{n+1} = V_{n}-\frac{\Delta t}{\gamma(V_{n})}\left[\left(\frac{\partial \cH}{\partial V}\right)_n + \frac{1}{\beta \gamma(V_{n})}\frac{d\gamma(V_n)}{dV}\right] + \sqrt{\frac{2\Delta t}{\beta \gamma(V_{n})}} \,z^V_n. \\
    \end{split}
\end{gather}

In particular, we choose $  V_{\b{r}}^{\star} =V(s_{\b{r},n})$ with $s_{\b{r},n}=(1-u_{\b{r}})t_n+u_{\b{r}}t_{n+1}$ for some $u_{\b{r}}\in(0,1]$. 
By Eq. \eqref{eq:rpV_general},
\begin{gather*}
     V(s_{\b{r},n})\approx V_n -\frac{u_r \Delta t}{\gamma(V_n)} \left[\left(\frac{\partial \cH}{\partial V}\right)_n+\frac{1}{\beta \gamma(V_n)}\frac{d\gamma(V_n)}{dV} - \beta^{-1}\right] + \sqrt{\frac{2}{\beta \gamma(V_n)}} \left(W_{s_{\b{r},n}}-W_{t_n}\right), 
\end{gather*}
Comparing Scheme \eqref{eq:rpV_temp} with Eq. \eqref{eq:rpV_general}, we desire to leading order the following holds
\begin{gather}
\frac{V_{n+1}-V_n}{3V_{\b{r}}^{\star}} \approx \frac{V_{n+1}-V_n}{3V_n} - \frac{2}{9\beta V_{n}^2\gamma(V_n)}\Delta t.
\end{gather}
Substituting $V_{\b{r}}^{\star}$ into the above equation, one thus requires the following
\begin{gather*}
    \begin{split}
        &3V_n^2\left(-\Delta t\left[\left(\frac{\partial \cH}{\partial V}\right)_n+\frac{1}{\beta \gamma(V_{n})}\frac{d\gamma(V_n)}{dV}\right] + \sqrt{2/\beta\gamma(V_n)}\left(W_{t_{n+1}}-W_{t_n}\right)\right)\\
        \approx &\left(V_n -\frac{u_{\b{r}}\Delta t}{\gamma(V_n)} \left[\left(\frac{\partial \cH}{\partial V}\right)_n+\frac{1}{\beta \gamma(V_n)}\frac{d\gamma(V_n)}{dV} - \beta^{-1}\right] + \sqrt{\frac{2}{\beta \gamma(V_n)}} \left(W_{s_{\b{r},n}}-W_{t_n}\right)\right)\\
        &\times \left[3V_n\left(-\Delta t\left[\left(\frac{\partial \cH}{\partial V}\right)_n + \frac{1}{\beta \gamma(V_{n})}\frac{d\gamma(V_n)}{dV}\right] + \sqrt{\frac{2}{\beta \gamma(V_n)}}\left(W_{t_{n+1}}-W_{t_n}\right)\right) - \frac{2\Delta t}{\beta\gamma(V_n)}\right].
    \end{split}
\end{gather*}
Since $\left(W_{s_{\b{r},n}}-W_{t_n}\right)\left(W_{t_{n+1}}-W_{t_n}\right)\to u_{\b{r}}\Delta t$, on thus finds 
\begin{gather}
u_{\b{r}}=1/3.
\end{gather}
 Therefore, we choose $V_{\b{r}}^{\star}=V_{n+1/3}$. Similarly, one can choose $V_{\b{p}}^{\star}=V_{n+2/3}$ for the second equation to be consistent.

The scheme \eqref{eq:rpV_temp} is interesting in the sense that it is consistent for both $M>0$ and $M\to 0$.
In this sense, the discretization in $\b{r}$ and $\b{p}$ can preserve the asymptotics.

These two schemes here are interesting in their own sense. The disadvantages are that they are only first order and there is no theoretic guarantee of the positivity of the volume. In the next section, we will choose a suitable friction function $\gamma(V)$ and propose our scheme based on the equations of motion.

\section{A second-order Langevin sampler for NPT ensemble preserving positive volume}\label{sec:second-order}

In this section, we propose a second-order Langevin sampler for the NPT ensemble based on the result in Theorem \ref{thm:invariant_measure}, which ensures that the volume of the periodic box is always positive throughout the simulation. 
In particular, the equations of motion would be \eqref{eq:rpLV_general} with the equation for $\epsilon$ specifically given by \eqref{eq:eps}. As mentioned, the key observation is that the noise in \eqref{eq:eps} is additive, and our approach will be based on the operator splitting \cite{lapidus1981generalization, leimkuhler2013robust}. In fact, with multiplicative noise, the Taylor expansion of the conditional expectation of test functions would involve derivative terms of the diffusion coefficient, making it challenging to design a second-order scheme. 

Note that Eq. \eqref{eq:rpLV_general} has the following three main parts, which represent different physics in the equations: 
\begin{enumerate}[(1)]
    \item Hamiltonian dynamics
        \begin{flalign}\label{eq:hamiltonian}
        &\text{(H)}& 
       \begin{cases}
         &  d\b{r} = \b{m}^{-1}\b{p}\,dt,\\
          &  d\b{p} = -\nabla_{\b{r}}U\,dt.
       \end{cases}
& \hspace{18em} 
    \end{flalign}

    \item Thermostat
    \begin{flalign}\label{eq:thermostat}
        &\text{(T)}& 
          d\b{p} = -\b{\gamma p}\,dt + \sqrt{2\b{m\gamma} \beta^{-1}}\,d\b{W}.
        & \hspace{12em} 
           \end{flalign}

    \item Barostat and stochastic rescaling
    \begin{flalign}\label{eq:barostat}
        &\text{(B)}& 
          \begin{cases}
            &d\epsilon = -\lambda\left(V\frac{\partial \cH}{\partial V} - \beta^{-1}\right)dt + \sqrt{2\lambda\beta^{-1}} dW_V,\\
            &d\b{r} = \b{r}\frac{dL}{L},\, d\b{p} = \b{p}\frac{dL^{-1}}{L^{-1}}, \text{ with } L = e^{\epsilon/3}.
        \end{cases}
 & \hspace{6em} 
    \end{flalign}
     
\end{enumerate}

The Barostat and stochastic rescaling part (i.e. (B) ) looks complicated at the first glance but it in fact has a simple dynamics due to the following observation.
\begin{lemma}\label{lem:ds=0}
    The reduced coordinates $\b{s} = L^{-1}\b{r}$ and the corresponding momenta $\b{p}^{\b{s}} = L\b{p}$ are invariant under the third part, Eq. \eqref{eq:barostat}, which means
    \begin{gather}\label{eq:ds=0}
        d\b{s} = \b{0}, \quad d\b{p}^{\b{s}} = \b{0}.
    \end{gather}
\end{lemma}

\begin{proof}
The conclusion can be verified by applying the It\^o's formula. In fact, 
\begin{gather*}
d\b{s} = \b{r}dL^{-1} + L^{-1}d\b{r} + d[L^{-1}, \b{r}]
            =\b{r}\left(-L^{-2}dL + L^{-3}d[L]\right) + L^{-1}\b{r}\frac{dL}{L} -L^{-3}d[L]\b{r} = 0.
\end{gather*}
    where $[X]$ denotes the quadratic variation of $X$, and $[X, Y] = [\frac{1}{2}(X+Y)] - [\frac{1}{2}(X-Y)]$ is the quadratic covariation of $X$ and $Y$. 
 Verification of $d\b{p}^s=0$ is similar.
  \end{proof}
 Hence, the evolution \eqref{eq:barostat} reduces to the dynamics of the barostat with $\b{s}$ and $\b{p}^s$ fixed.  
  This dynamics has no exact solution. However, since it is a single SDE with an additive noise, there are many convenient second order schemes to solve it. In particular, we implement the following prediction-correction scheme, which is a second-order Runge-Kutta scheme.
\begin{gather}\label{eq:eps_pc}
    \begin{small}
    \begin{cases}
        &\epsilon_{n+1}^{\text{EM}} = \epsilon_{n} -\lambda\left(V_n\left.\frac{\partial \cH}{\partial V}\right|_{V_n} - \beta^{-1}\right)\Delta t + \sqrt{2\Delta t \lambda\beta^{-1}} \, z^V_{n},\\
        &\epsilon_{n+1} = \epsilon_{n} - \lambda\left(V_{n}\left.\frac{\partial \cH}{\partial V}\right|_{V_n} + V_{n+1}^{\text{EM}}\left.\frac{\partial \cH}{\partial V}\right|_{V_{n+1}^{\text{EM}}} - 2\beta^{-1}\right)\frac{\Delta t}{2} + \sqrt{2\lambda\Delta t\beta^{-1}} \, z^V_{n}.
    \end{cases}
    \end{small}
\end{gather}
Here, $\left.\frac{\partial \cH}{\partial V}\right|_{V_{n+1}^{\text{EM}}}$ indicates that $\cH$ is viewed as the function of $(\b{s}, \b{p}^s, V)$ and $V$ is updated to $V_{n+1}^{\text{EM}}$. In particular, the following update shall be used in the implementation if $\cH$ is given as a function of $(\b{r}, \b{p}, V)$
\[
\b{r}_{n+1}^{\text{EM}} \leftarrow \frac{L_{n+1}^{\text{EM}}}{L_{n}}\b{r}_n,
\quad \b{p}_{n+1}^{\text{EM}} \leftarrow \frac{L_{n}}{L_{n+1}^{\text{EM}}}\b{p}_n.
\]
After $\epsilon_{n+1}$ is obtained, the stochastic rescaling for $\b{r}$ and $\b{p}$ is then solved as follows: 
\begin{gather}\label{eq:scaling}
    \b{r}_{n+1} \leftarrow \frac{L_{n+1}}{L_{n}}\b{r}_n, \,\b{p}_{n+1} \leftarrow \frac{L_{n}}{L_{n+1}}\b{p}_n.
\end{gather}
We remark that, since the noise in Eq. \eqref{eq:barostat} is additive, Scheme \eqref{eq:eps_pc} achieves second-order weak convergence without the need to evaluate derivatives of the diffusion coefficient, and only one standard normal variable is sampled at each time step.

The other two parts are relatively standard in literature. For example, the thermostat \eqref{eq:thermostat}, which evolves $\b{p}$ over time $\Delta t$ under the Ornstein-Uhlenbeck process $d\b{p}=-\b{\gamma}\b{p}\,dt + \sqrt{2\beta^{-1}\b{m\gamma}}\,dW$, can be explicitly integrated as follows:
\begin{gather*}
    \b{p}(t+\Delta t) = e^{-\Delta t\b{\gamma}}\b{p}(t) + \sqrt{(1-e^{-2\Delta t\b{\gamma}})\beta^{-1}\b{m}}\,\b{z},
\end{gather*}
where $\b{z}$ is a $d$-dimensional vector with each component drawn from a normal distribution with zero mean and unit variance.

For the Hamiltonian dynamics \eqref{eq:hamiltonian}, we again split it into 
 \begin{flalign}\label{eq:displacement}
        &\text{(D)}& 
          \dot{\b{r}}=\b{m}^{-1}\b{p},
        & \hspace{15em} 
\end{flalign}
and
 \begin{flalign}\label{eq:force}
        &\text{(F)}& 
         \dot{\b{p}}=-\nabla_{\b{r}}U.
        & \hspace{15em} 
\end{flalign}
Here the labels "$D$" and "$F$" are short for "displacement" and "force". The benefit of this splitting is that the simple Euler's method is exact for both dynamics. For example, the symplectic Euler's method can be viewed as the solver associated with this splitting.

The law of $\b{X}:=[\b{r}^T, \b{p}^T, \epsilon]^T$, denoted by $\mu_t:= \mu(\b{r}, \b{p}, \epsilon; t)$, satisfies the following Fokker-Planck equation
\begin{gather}\label{eq:FokkerPlank}
    \partial_t\mu_t = \mathcal{L}^{*}\mu_t,
\end{gather}
where $\cL$ represents the generator of the process and $\cL^*$ is the adjoint operator of $\cL$.  The solution to Eq. \eqref{eq:FokkerPlank} can be formally written as $\mu_t = e^{t\mathcal{L}^{\star}}\mu_0$. 
Let the generators for $(D)$, $(F)$, $(T)$ and $(B)$ be $\cL_T$, $\cL_F$, $\cL_T$ and $\cL_B$ respectively. The law of the process under these systems of equations would evolve according to the semigroups $e^{t\cL_i^*}$ where $i=D,F,T,B$.
With these notations, we propose the following splitting method with propagator
\begin{gather}\label{eq:operator_splitting}
    e^{\frac{\Delta t}{2}\cL_{T}^*} e^{\frac{\Delta t}{2}\cL_{F}^*} e^{\frac{\Delta t}{2}\cL_{D}^*} e^{\Delta t\cL_{B}^*}e^{\frac{\Delta t}{2}\cL_{D}^*} e^{\frac{\Delta t}{2}\cL_{F}^*} e^{\frac{\Delta t}{2}\cL_{T}^*} := e^{\Delta t\cL_{s}^*}. 
\end{gather}
The corresponding algorithm is displayed in Algorithm \ref{alg:2ndLangevinSampler}.

\begin{algorithm}[H]
\caption{(Second order Langevin sampler)}\label{alg:2ndLangevinSampler}
\begin{algorithmic}[1]
\State Conduct the thermostat for half a step:
    \begin{gather}
     \widetilde{\b{p}_{n+\frac{1}{2}}} = e^{-\frac{\Delta t}{2}\b{\gamma}}\b{p}_{n} + \sqrt{(1-e^{-\Delta t\b{\gamma}})\beta^{-1}\b{m}}\,\b{z}_n^{(1)}.
    \end{gather}
\State Evolve the Hamiltonian dynamics by updating $\b{p}$ followed by $\b{r}$ for half-a-step: 
    \begin{gather}
       \b{p}_{n+\frac{1}{2}} = \widetilde{\b{p}_{n+\frac{1}{2}}}-\frac{\Delta t}{2}\nabla_{\b{r}}U(\b{r}_n),\quad 
       \b{r}_{n+\frac{1}{2}} = \b{r}_n + \frac{\Delta t}{2}\b{m}^{-1}\b{p}_{n+\frac{1}{2}}.
    \end{gather}
    
\State Conduct the barostat for a full step using \eqref{eq:eps_pc} and \eqref{eq:scaling}: 
    \begin{gather}
        \begin{cases}
            \,\,V_{n+1} = e^{\epsilon_{n+1}}, \quad L_{n+1} = V_{n+1}^{\frac{1}{3}}, \\
            \,\,\b{r}_{n+\frac{1}{2}}^{\star} = \frac{L_{n+1}}{L_{n}}\b{r}_{n+\frac{1}{2}}, \quad \b{p}_{n+\frac{1}{2}}^{\star} = \frac{L_{n}}{L_{n+1}}\b{p}_{n+\frac{1}{2}}.
        \end{cases}
    \end{gather}
\State Evolve the Hamiltonian dynamics by updating $\b{r}$ followed by $\b{p}$ for half-a-step: 
    \begin{gather}
    \b{r}_{n+1} = \b{r}_{n+\frac{1}{2}}^{\star} + \b{m}^{-1}\b{p}_{n+\frac{1}{2}}^{\star}\frac{\Delta t}{2},\quad \widetilde{\b{p}_{n+1}}=\b{p}_{n+\frac{1}{2}}^{\star}-\nabla_{\b{r}}U(\b{r}_{n+1})\frac{\Delta t}{2}.
    \end{gather}
\State Conduct the thermostat for half a step:
    \begin{gather}
       \b{p}_{n+1} = e^{-\frac{\Delta t}{2}\b{\gamma}}\widetilde{\b{p}_{n+1}} + \sqrt{(1-e^{-\Delta t\b{\gamma}})\beta^{-1}\b{m}}\,\b{z}_n^{(2)}.
    \end{gather}
\end{algorithmic}
\end{algorithm}

Since the numerical solutions are all constructed explicitly, we easily conclude the following.
\begin{proposition}
The numerical solution $(\b{r}_{n}, \b{p}_n, \epsilon_n)$ is well-defined for all $n\ge 0$ and thus $V_n=e^{\epsilon_n}$ is positive for all $n\ge 0$.
\end{proposition}

Next, we investigate the order of the accuracy for the splitting method in \eqref{eq:operator_splitting}.
In fact, the order of the propagators $e^{\Delta t\cL^*_i}$ is symmetric, and only the barostat part is solved with a second order weak scheme while others are solved exactly. Let $\mathcal{P}_B^*$ be the operator that approximates $e^{\Delta t \cL_B^*}$ in the numerical scheme for the evolution of the law. Then, the evolution of the law after one-step of the numerical method is given by
\begin{gather}\label{eq:numericalpropagator}
\mathcal{S}^*=e^{\frac{\Delta t}{2}\cL_{T}^*} e^{\frac{\Delta t}{2}\cL_{F}^*} e^{\frac{\Delta t}{2}\cL_{D}^*} \mathcal{P}_B^* e^{\frac{\Delta t}{2}\cL_{D}^*} e^{\frac{\Delta t}{2}\cL_{F}^*} e^{\frac{\Delta t}{2}\cL_{T}^*}.
\end{gather}

We then conclude that the method has the second order weak accuracy.
\begin{proposition}\label{the:2ndLangevinSampler}
Suppose that the potential function $U$ is a smooth function.   The splitting method proposed in Algorithm \ref{alg:2ndLangevinSampler} is a second order weak scheme for solving the corresponding  SDE \eqref{eq:rpLV_general} with the particular choice of friction \eqref{eq:friction}.
\end{proposition}

We first note a well-known fact in for splitting, stating that a symmetry propagator splitting has second-order accuracy. 
 \begin{gather}\label{eq:BCHlemma}
        e^{\frac{\Delta t}{2}\cL_1} \cdots e^{\frac{\Delta t}{2}\cL_{n-1}} e^{\Delta t\cL_n} e^{\frac{\Delta t}{2}\cL_{n-1}} \cdots e^{\frac{\Delta t}{2}\cL_1} = e^{\Delta t\sum_{i=1}^n\cL_i} + \mathcal{O}(\Delta t^3),
    \end{gather}
    where we assume that $e^{s\cL_i}$ is a strongly continuous semigroup.
 In fact,  for the case $n=2$, applying the BCH formula yields
    \begin{gather}
    \begin{split}
        & e^{\frac{\Delta t}{2}\cL_1}e^{\Delta t\cL_2}e^{\frac{\Delta t}{2}\cL_1}\\
        = & \exp\left\{\frac{\Delta t}{2}\cL_1 + \Delta t\cL_2 + \frac{1}{2} [\frac{\Delta t}{2}\cL_1, \Delta t\cL_2] + \mathcal{O}(\Delta t^3)\right\}e^{\frac{\Delta t}{2}\cL_1}\\
        = & \exp\left\{\frac{\Delta t}{2}\cL_1 + \Delta t\cL_2 + \frac{\Delta t^2}{4} [\cL_1, \cL_2] +\frac{\Delta t}{2}\cL_1 + \frac{\Delta t^2}{2}[\frac{1}{2}\cL_1 + \cL_2, \frac{1}{2}\cL_1] + \mathcal{O}(\Delta t^3)\right\}\\
        = & e^{\Delta t(\cL_1+\cL_2)} + \mathcal{O}(\Delta t^3),
    \end{split}
    \end{gather}
    where $[X,Y] = XY-YX$ is the Lie bracket. The second-order terms of $\Delta t$ offset due to the symmetry. 
    For the case $n>2$, Eq. \eqref{eq:BCHlemma} can be established by recursion on $n$.   
 With this fact, we actually have the following claim.
\begin{lemma}
Consider the generators $\cL_i$ in \eqref{eq:operator_splitting}.  For distribution $\rho_0$ that has all orders of moments of $(\b{r},\b{p}, V)$, and any test function $\varphi\in C_p^{\infty}$, for $\Delta t$ small enough, the following hold
    \begin{gather}
|\langle e^{\Delta t \cL_s}\varphi-e^{\Delta t\cL}\varphi, \rho_0\rangle|\le C_{\varphi}\Delta t^3.
    \end{gather}
where the constant $C_{\varphi}$ depends on $\varphi$ and $\rho_0$ but independent of $\Delta t$.
\end{lemma}
\begin{proof}
First note that the coefficients of the generators are smooth (see \eqref{eq:VpartialH} for the formula of $V\partial\cH/\partial V$).
The proof is actually very straightforward by the BCH formula, by noting that all the operators (including those in the remainder $O(\Delta t^3)$) acting on $\varphi\in C_p^{\infty}$ gives a function that is integrable against $\rho_0$.
\end{proof}

With the observation, Proposition \ref{the:2ndLangevinSampler} then follows directly since the only difference is that
$e^{\Delta t\cL_{\text{B}}}$ is approximated by $\mathcal{P}_B$ with one-step third order weak error.
Then, fixing any $\rho_0$ that has finite moments of any order, one then has (note that $\mathcal{S}$ is the dual of \eqref{eq:numericalpropagator})
\begin{gather*}
|\langle \mathcal{S}\varphi-e^{\Delta t \cL_s}\varphi, \rho_0\rangle|\le C_{\varphi}\Delta t^3.
\end{gather*}

To prove that the method converges on finite interval $[0, T]$ with weak order $2$, one needs to control the tails of the numerical density or the density for the time continuous equations. 
Moreover, what we care about is the error of the invariant measure. This needs the ergodicity of the Markov chain generated by the numerical propagator.
Together with some mixing properties of the numerical propagator, the weak convergence order of the invariant measure can be shown to be 2 as well. We find the rigorous study of these theoretic properties needs significant extra effort and we would like to leave them for future study.

\section{Numerical Examples}\label{sec:experiments}

In this section, a free gas model, an artificial interacting particle system and the Lennard-Jones fluid are used as test cases to validate the effectiveness and applicability of our second-order Langevin sampler across different scenarios. First, the free gas model facilitates straightforward testing under ideal conditions, serving as a baseline for evaluating algorithm performance. Second, an interacting particle system with artificial interaction further demonstrates the second-order convergence of our algorithm. Finally, the Lennard-Jones fluid, a standard model in molecular simulations characterized by its interacting potential, tests the robustness of our algorithm in more complex, realistic scenarios and higher dimensional spaces. 

Taking all the periodic boxes into account, the potential $U$ is given by \eqref{eq:potentialU}. 
In this setting, it is not convenient to use the viral formula \eqref{eq:inspressure1} to compute the pressure. Instead, we compute the partial derivative of the energy with respect to $V$ by fixing $\b{s}$ and $\b{p}^s$ directly.
\begin{gather}
    \begin{split}
        P = & -\frac{\partial}{\partial V}\left(U(V^{1/3}\b{s}; V) - \frac{1}{2}V^{-2/3}\b{p}^{\b{s},T}\b{m}^{-1}\b{p}^{\b{s}}\right)\\
           =& \frac{1}{3V}\b{p}^T\b{m}^{-1}\b{p} - \frac{1}{3V}\sum_{1\leq i<j\leq N} \sum_{\b{n}} \left(\b{r}_{ij} + V^{1/3}\b{n}\right)^T\nabla\Phi\left(\b{r}_{ij} + V^{1/3}\b{n}\right).
    \end{split}
\end{gather}

\subsection{Free gas}

In this example, we consider the simplest model: the free gas with $\Phi=0$. In this setting, the $N$ particles move independently within the simulation box, with neither interactions between particles nor external forces acting upon them. The dynamics is then governed by the thermostat and the barostat.  The marginal distribution of the stationary distribution with respect to $V$ is given by
\begin{gather}\label{eq:rhoV_toy}
    \rho_V(V) = \frac{\left(\beta P_0\right)^{N+1}}{N!}V^Ne^{-\beta P_0 V}.
\end{gather}

\begin{figure}[!htbp]
    \centering
    \subfigure{\includegraphics[width=0.48\textwidth]{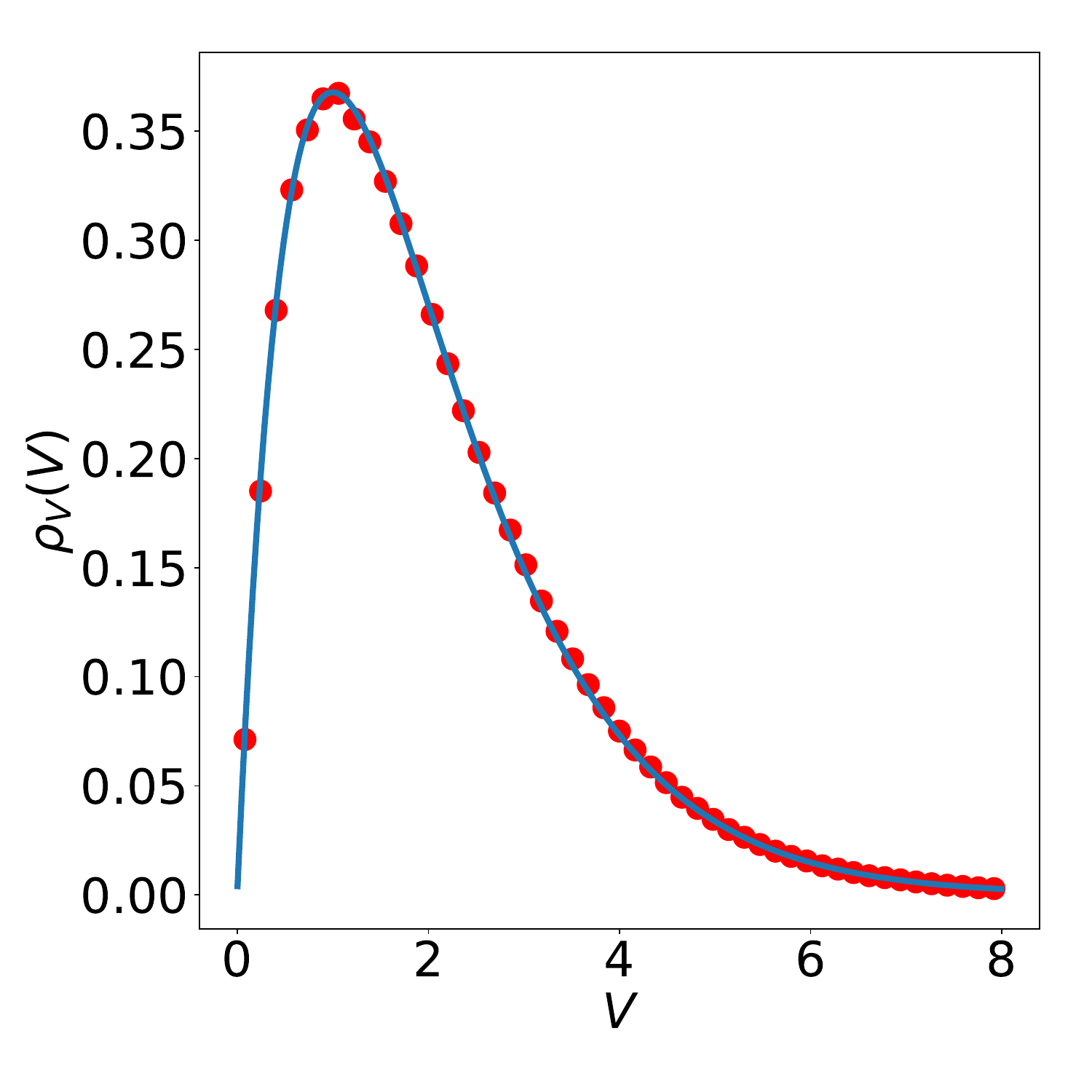}}
    \hfill
    \subfigure{\includegraphics[width=0.48\textwidth]{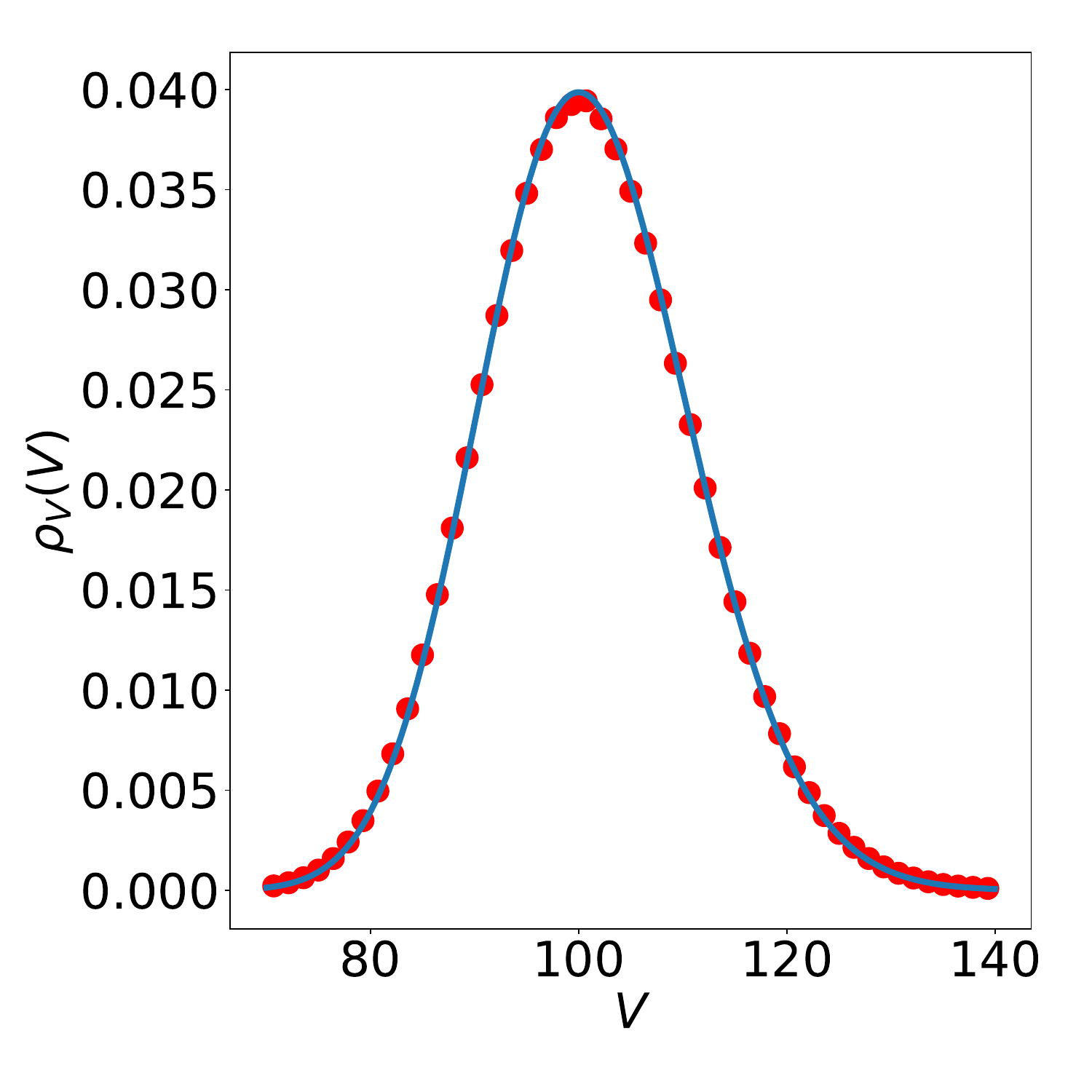}}
    \caption{Marginal distribution of the stationary distribution concerning $V$. Left panel: the numerical results of the experiment with $N=1$ particle. Right panel: the numerical results of the experiment with $N=100$ particles. In both panels. the red circles are the empirical density obtained by $10^7$ iterations with a time step of $\Delta t=10^{-3}$ while the blue curve is the exact density \eqref{eq:rhoV_toy}. The empirical density is perfectly overlapping with the exact one. }
    \label{fig:rhoV_toy}
\end{figure}

In the simulations, we set $P_0=1.0$ and $\beta=1.0$. We perform $10^7$ iterations with a time step of $\Delta t=10^{-3}$
to sample from the target ensemble, using particle numbers $N=1$ and $N=100$, respectively. Samples are collected at every iteration, with no samples discarded at the beginning of the iteration. 
Figure \ref{fig:rhoV_toy} illustrates the numerical solution of the marginal distribution with respect to $V$, demonstrating that the empirical distribution (red circles) is indistinguishable from the exact distribution (blue curve).

\begin{figure}[!htbp]
    \centering
    \subfigure{\includegraphics[width=0.48\textwidth]{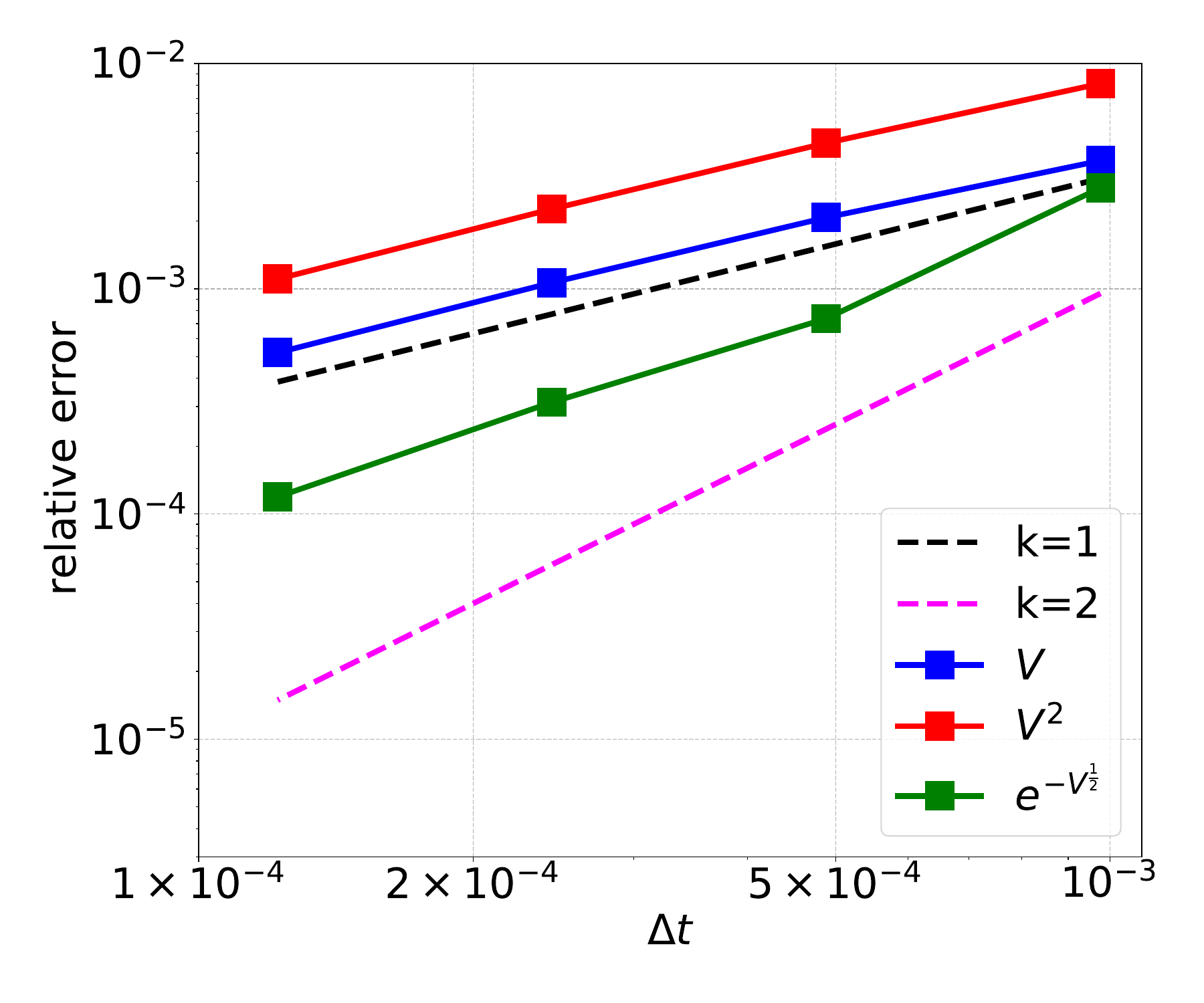}}
    \hfill
    \subfigure{\includegraphics[width=0.48\textwidth]{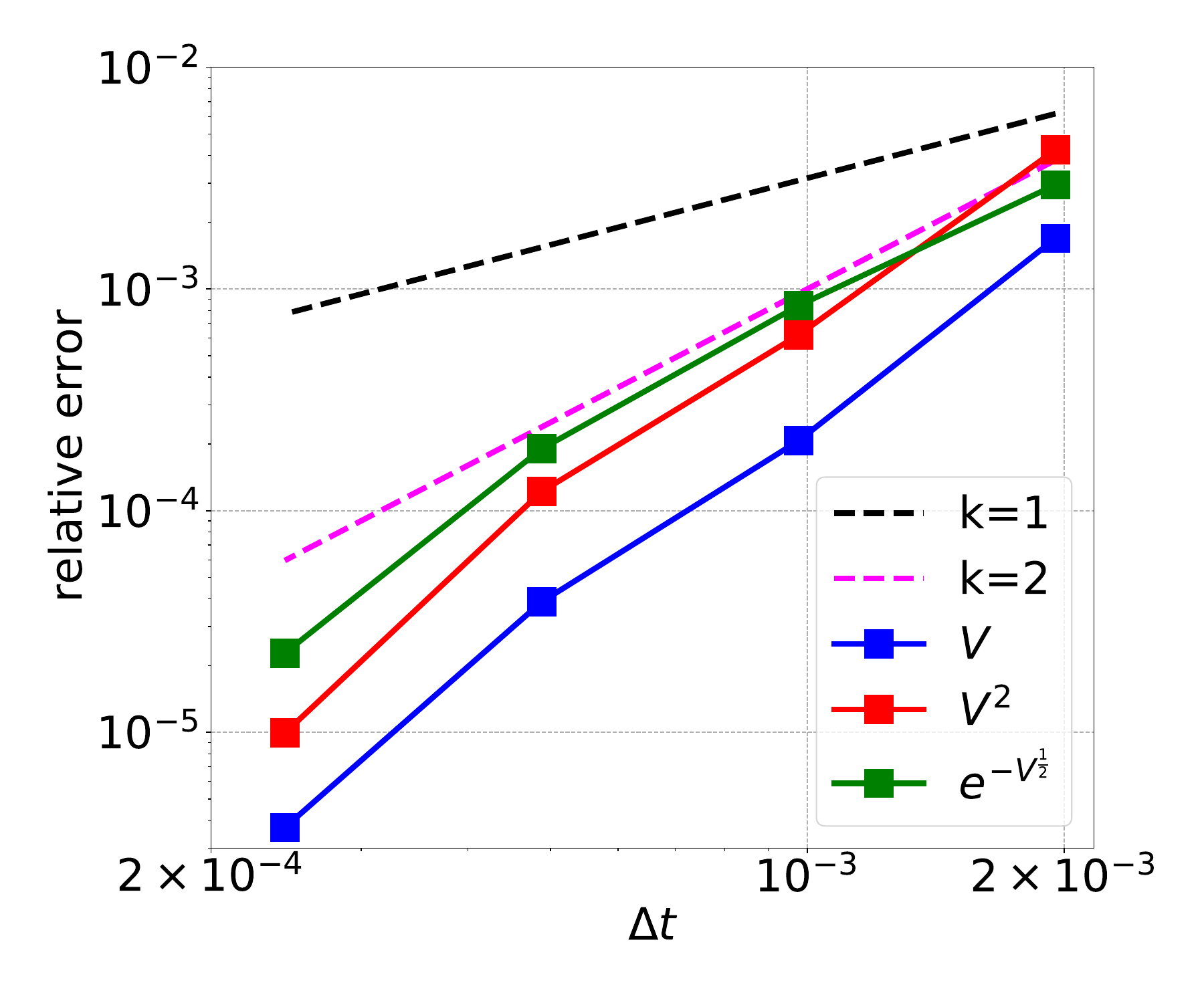}}
    \caption{Relative error verse time step with the left panel for Scheme \eqref{eq:EM_Vrp} and the right panel for Algorithm \ref{alg:2ndLangevinSampler}. The numerical results obtained by taking the time step $\Delta t_{\text{ref}}=2^{-14}$ serve as the reference solution. The solid curves of various colors present the relative errors in different test functions, as shown in the legend. The magenta dotted curve indicates the second-order convergence rate. }
    \label{fig:error_free-gas}
\end{figure}

In the following, we numerically verify that our proposed scheme has second-order accuracy. We set the number of particles to $N=400$ and fix the total evolution time to $T_{\text{E}}=1.0$. The numerical result obtained with a time step of $\Delta t_{\text{ref}}=2^{-14}$ is used as the reference solution. Letting $\varphi$ denote the test function used to measure the weak order of convergence, the relative error is calculated as:
\begin{gather}\label{eq:relError}
    \left|\frac{1}{K}\sum_{k=1}^{K} \varphi^{(\ell)}_{T_{\text{E}},k} - \frac{1}{K}\sum_{k=1}^{K} \varphi^{(0)}_{T_{\text{E}},k}\middle|\right/\frac{1}{K}\sum_{k=1}^{K} \varphi^{(0)}_{T_{\text{E}},k}
\end{gather}
where $\varphi^{(\ell)}_{T_{\text{E}},k}$ denotes the $k$-th sample of the numerical approximation of $\varphi$ at time $T_{\text{E}}$, computed by taking the time step $\Delta t=2^{-\ell}$ with $\ell = 9,\ldots, 13$. There are $K=10^6$ samples in total, and the test functions chosen are $\varphi=V, V^2$ and $e^{-V^{1/2}}$. Figure \ref{fig:error_free-gas} presents the relative errors in these test functions, with the left panel corresponding to Scheme \eqref{eq:EM_Vrp} and the right panel corresponding to Algorithm \ref{alg:2ndLangevinSampler}. One can see from Figure \ref{fig:error_free-gas} that Scheme \eqref{eq:EM_Vrp} converges in first order while Algorithm \ref{alg:2ndLangevinSampler} converges second order.

\subsection{An artificial interacting particle system}
In the second toy example, we consider interaction between particles in the form of:
\begin{gather}
    \Phi(\b{r}_{ij}) = \frac{1}{1+\left|\b{r}_{ij}\right|^4},
\end{gather}
where $\b{r}_{ij} = \b{r}_i-\b{r}_j$ is the three-dimensional displacement from locations $\b{r}_j$ to $\b{r}_i$. Taking the images in the periodic boxes into account, each particle $i$ interacts with not only all the other particles $j\neq i$ in the current box but also their periodic images. Therefore, the potential of the whole system is
\begin{gather}
    U(\b{r}; V) = \sum_{1\leq i<j\leq N}w_iw_j \sum_{\b{n}} \Phi(\b{r}_{ij}+V^{1/3}\b{n}),
\end{gather}
where $w_i$ denotes the weight of particle $i$, and $\b{n}\in\mathbb{Z}^3$ ranges over the three-dimensional integer column vectors. 
\begin{remark}
We remark that for the Coulomb potential, $\sum_{\b{n}} \Phi(\b{r}_{ij}+V^{1/3}\b{n})$ is divergent and one must use the charge neutral conditions to rearrange the summation to compute the total energy.
\end{remark}

We consider indistinguishable particles such that $w_i = w$ for $i = 1, \ldots, N$. Noting that the interaction potential gradually vanishes for as $|\b{r}|$ increases, we only consider the images in the surrounding boxes sharing faces with the original boxes. 
Therefore, the force is approximated by
\begin{gather}
    \b{F}_i \approx w^2\sum_{j\neq i} \sum_{\b{n}\in\mathcal{N}} \frac{4\left|\b{r}_{ij}+V^{1/3}\b{n}\right|^2\left(\b{r}_{ij}+V^{1/3}\b{n}\right)}{\left(1+\left|\b{r}_{ij}+V^{1/3}\b{n}\right|^4\right)^2}, 
\end{gather}
where $\mathcal{N} = \{\b{0}\}\cup \{\b{n}\in\mathbb{Z}^3: \left\|\b{n}\right\| = 1\}$, and the pressure is approximated by
\begin{gather}
    P \approx \frac{\b{p}^T\b{m}^{-1}\b{p}}{3V} + \frac{w^2}{3V}\sum_{1\leq i<j\leq N} \sum_{\b{n}\in\mathcal{N}}\frac{4\left|\b{r}_{ij}+V^{1/3}\b{n}\right|^4}{\left(1+\left|\b{r}_{ij}+V^{1/3}\b{n}\right|^4\right)^2}.
\end{gather}

Setting the number of particles to $N=10$, the pressure to $P_0=1.0$ and the temperature to $k_BT_0=2.0$, we are going to verify the second-order weak convergence rate again. Similar to the above example, we fix the total evolution time to $T_{\text{E}}=1.0$ and use the numerical result obtained with a time step of $\Delta t_{\text{ref}}=2^{-12}$ as the reference solution. Here, we use $V$, $P$, $\rho=N/V$ and $\sqrt{V}e^{-\sqrt{V}}$ as the test functions and we collect $K=10^6$ samples. The relative errors (defined in Eq. \eqref{eq:relError}) in the numerical solutions computed by taking the time step of $\Delta t = 2^{-\ell}, \ell = 7, \ldots, 11$ are presented in Figure \ref{fig:error_toy2}, which again exhibits the second-order convergence rate. For comparison, we also give the relative errors of Scheme \eqref{eq:EM_Vrp}, where the reference solution is given by the numerical solution of taking the time step $\Delta t_{\text{ref}} = 2^{-14}$, since Scheme \eqref{eq:EM_Vrp} cannot tolerate relatively large time step such as $2^{-8}$. 
\begin{figure}[!htbp]
    \centering
    \subfigure{\includegraphics[width=0.48\textwidth]{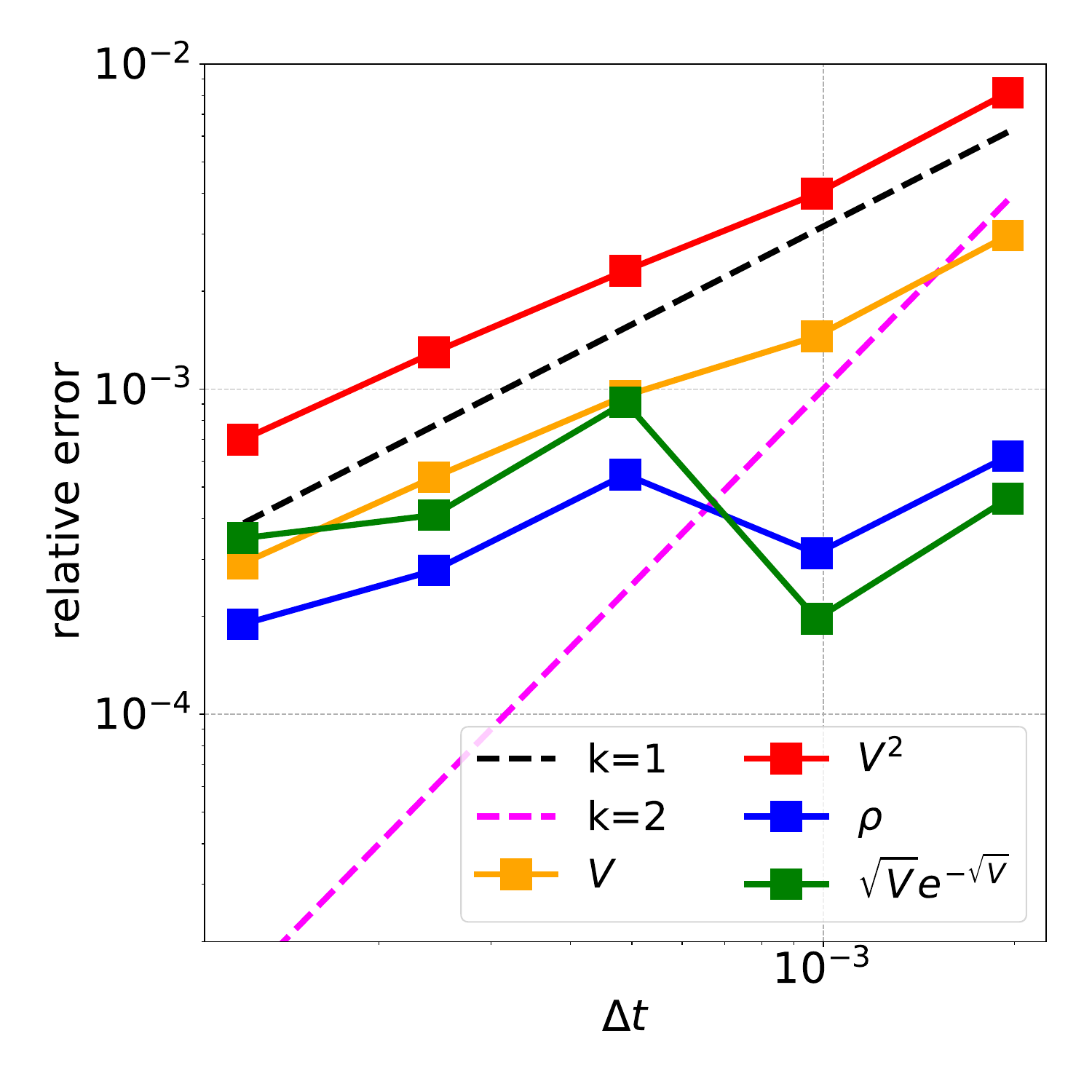}}
    \hfill
    \subfigure{\includegraphics[width=0.48\textwidth]{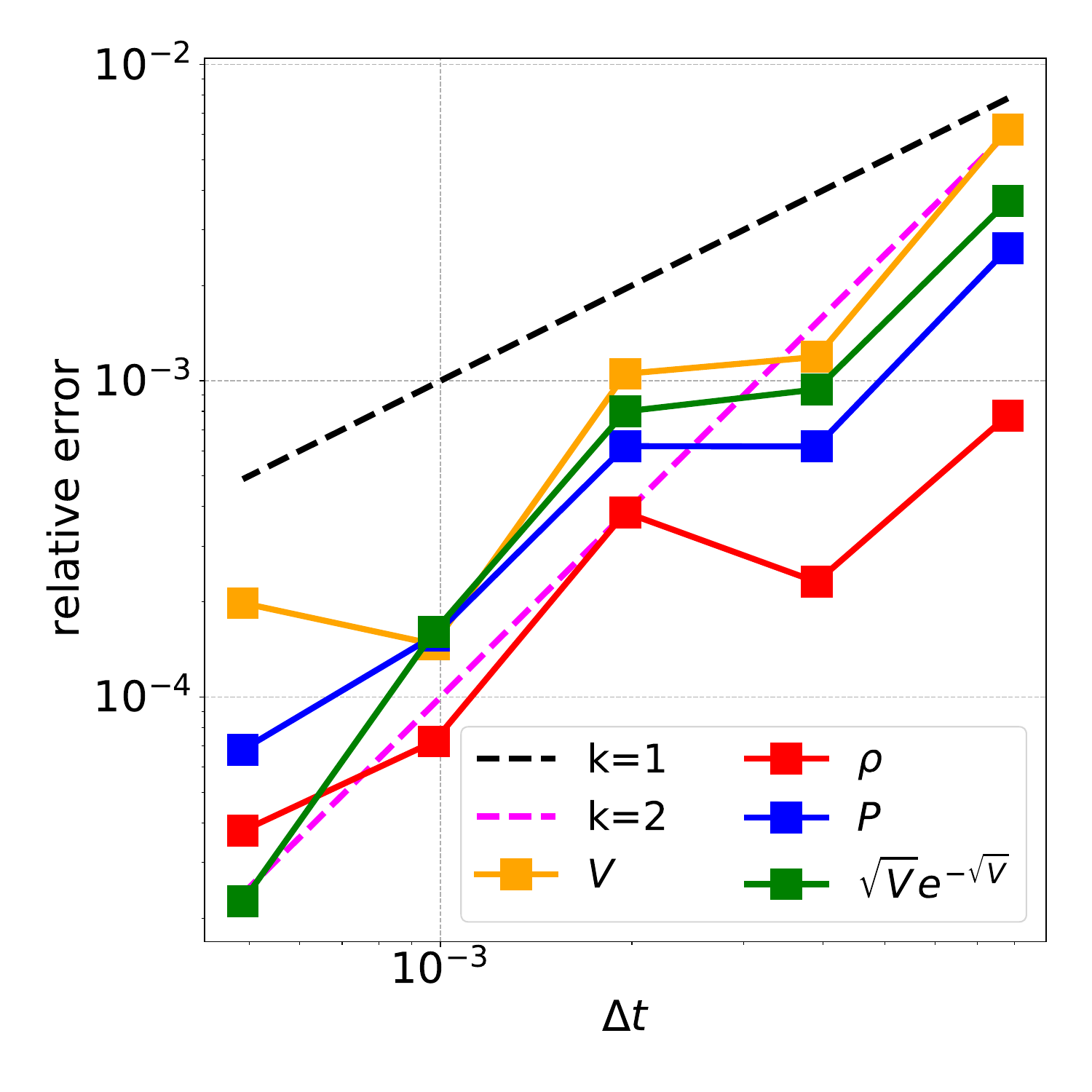}}
    \caption{Relative error verse time step. Left panel: results obtained by Scheme \eqref{eq:EM_Vrp}. The reference solution is given by the numerical result of taking time step$\Delta t_{\text{ref}} = 2^{-14}$. The right panel: results obtained by Algorithm \ref{alg:2ndLangevinSampler}. The reference solution is given by the numerical result of taking time step$\Delta t_{\text{ref}} = 2^{-12}$. For both panels, the solid curves of various colors present the relative errors in different test functions, as shown in the legend. The magenta dotted curve indicates the second-order convergence rate while the black dotted curve indicates the first-order convergence rate. }
    \label{fig:error_toy2}
\end{figure}

\subsection{Lennard-Jones fluid}
We now test our scheme on the simulation of a Lennard-Jones fluid in the NPT ensemble. In this scenario, the interaction between particles is modelled by the Lennard-Jones potential 
\begin{gather}\label{eq:potential_LJ}
    U_{\text{LJ}}(r_{ij}) = 4\left[\left(\frac{1}{r_{ij}}\right)^{12} - \left(\frac{1}{r_{ij}}\right)^{6}\right],
\end{gather}
where $r_{ij} = |\b{r}_i - \b{r}_j|$ denotes the distance between two particles located at $\b{r}_i$ and $\b{r}_j$. Noting that the Lennard-Jones potential diminishes significantly as the distance between particles increases, we employ a cutoff distance $r_c = \max\{L/2, 2.5\}$ with $L=V^{1/3}$ and ignore the interactions between particles at distances greater than $r_c$, which coincides with the setting in \cite[Section \uppercase\expandafter{\romannumeral3}.A]{bernetti2020pressure}. Therefore, the force and the pressure is approximated by
\begin{gather}
    \begin{split}
        \b{F}_i & \approx 24\sum_{j\neq i}\left(2\left|\b{r}_{ij}+V^{1/3}\b{n}_{ij}^{\star}\right|^{-14} - \left|\b{r}_{ij}+V^{1/3}\b{n}_{ij}^{\star}\right|^{-8}\right)\left(\b{r}_{ij}+V^{1/3}\b{n}_{ij}^{\star}\right),\\
        P & \approx \frac{\b{p}^T\b{m}^{-1}\b{p}}{3V} + \frac{8}{V}\sum_{1\leq i<j\leq N} \left(2\left|\b{r}_{ij}+V^{1/3}\b{n}_{ij}^{\star}\right|^{-12} - \left|\b{r}_{ij}+V^{1/3}\b{n}_{ij}^{\star}\right|^{-6}\right)\\
    \end{split}
\end{gather}
where $\b{n}_{ij}^{\star} = \arg\min_{\b{n}}\left|\b{r}_{ij}+V^{1/3}\b{n}\right|$ indicates the periodic box containing the image of $j$ such that the distance between $i$ and $j$ is minimized.

\begin{table}[]
    \centering
    \subtable[$N=5$]{
    \begin{small}
    \begin{tabular}{c|cccccccccc}
       $P_0$ &0.25 &0.5 &1.0 &2.0 &3.0 &4.0 &5.0 &6.0 &7.0 &8.0\\
       \hline
       $E_1$ &0.0009 &1.8E-5 &0.0004 &2.3E-5 &0.0003 &0.0005 &0.0003 &0.0003 &0.0002 &0.0001 \\
       \vspace{0.5em}
       $E_2$ &0.0103 &0.0107 &0.0109 &0.0112 &0.0112 &0.0112 &0.0113 &0.0114 &0.0114 &0.0114 \\
    \end{tabular}
    \end{small}
    }
    \subtable[$N=10$]{
    \begin{small}
    \begin{tabular}{c|cccccccccc}
       $P_0$ &0.25 &0.5 &1.0 &2.0 &3.0 &4.0 &5.0 &6.0 &7.0 &8.0 \\
       \hline
       $E_1$ &0.0004 &0.0004 &0.0011 &0.0008 &0.0007 &0.0004 &0.0004 &0.0004 &0.0004 &0.0003 \\
       \vspace{0.5em}
       $E_2$ &0.0160 &0.0156 &0.0154 &0.0152 &0.0152 &0.0151 &0.0151 &0.0150 &0.0150 &0.0150 \\
    \end{tabular}
    \end{small}
    }
    \subtable[$N=100$]{
    \begin{small}
    \begin{tabular}{c|cccccccccc}
       $P_0$ &0.25 &0.5 &1.0 &2.0 &3.0 &4.0 &5.0 &6.0 &7.0 &8.0 \\
       \hline
       $E_1$ &3.1E-4 &4.1E-4 &2.7E-4 &1.7E-4 &1.3E-4 &9.6E-5 &7.9E-5 &5.8E-5 &4.6E-5 &4.0E-5 \\
       \vspace{0.5em}
       $E_2$ &0.0313 &0.0308 &0.030 &0.029 &0.027 &0.0260 &0.0245 &0.0228 &0.0210 &0.0195 \\
    \end{tabular}
    \end{small}
    }
    \caption{Relative errors in the statistics of pressure virial theorems under different pressure $P_0$, where $E_1$ and $E_2$ is given by Eq. \eqref{eq:E1E2}. Here, all the experiments are conducted under time step size $\Delta t = 10^{-5}$ with $10^8$ iterations. }
    \label{tab:LJ_error}
\end{table}

\begin{figure}[!htbp]
    \centering    
    \includegraphics[width=\textwidth]{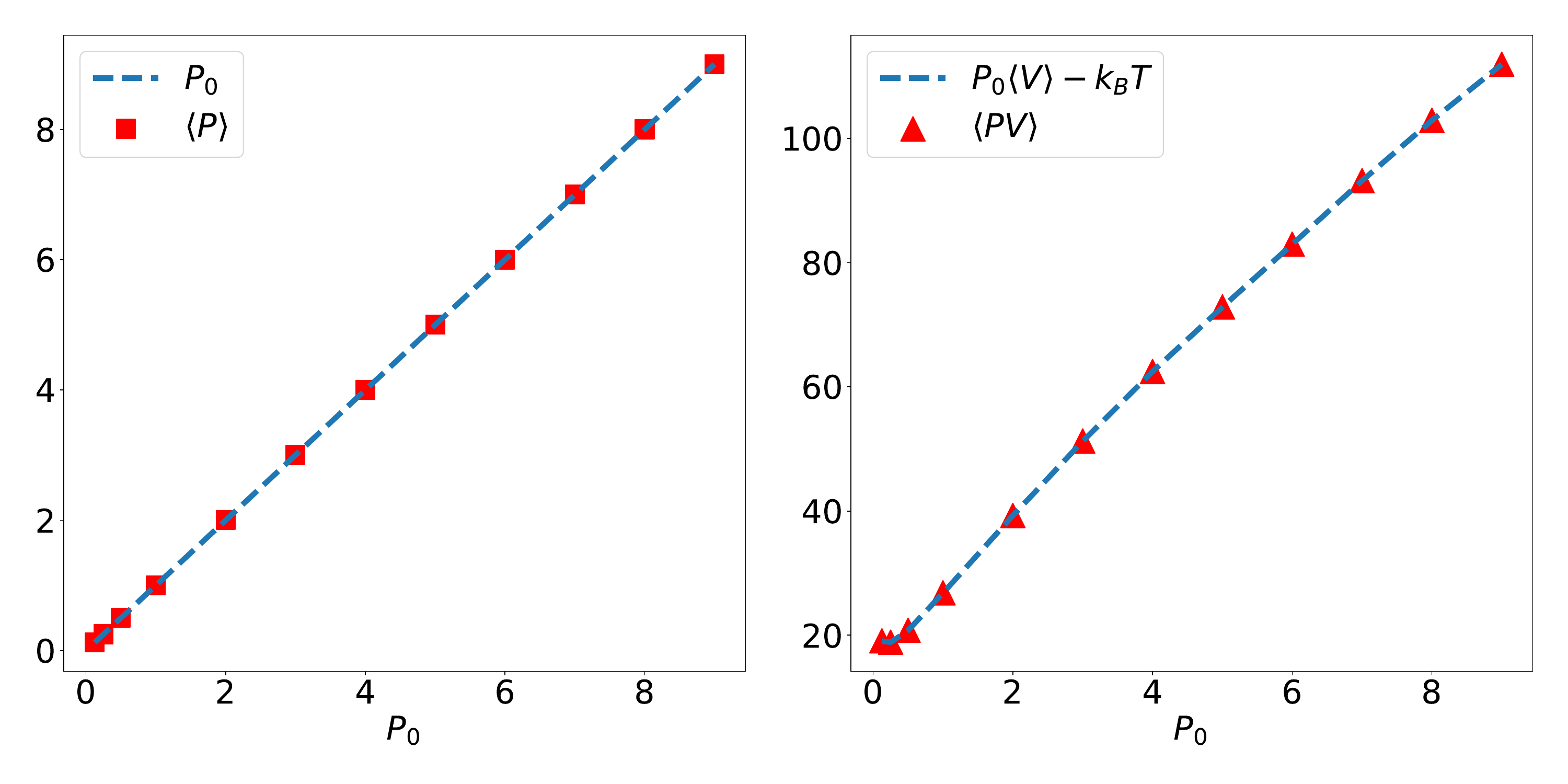}
    \caption{Left panel: Ensemble average $\langle P\rangle$ under different pressure $P_0$. Right panel: Ensemble average $\langle PV\rangle$ under different pressure $P_0$. For both panels, the numerical results are from simulations of the system with $N=10$ particles. The ensembles averages (red squares) are finely overlapped with the theoretical values (blue dotted line), which are $P_0$ for $\langle P\rangle$ and $P_0\langle V\rangle-\beta^{-1}$ for $\langle PV\rangle$, respectively, according to the pressure virial theorems Eq. \eqref{eq:P-virial}. }
    \label{fig:LJ_10_P-virial}
\end{figure}

\begin{figure}[!htbp]
    \centering
    \includegraphics[width=0.6\textwidth]{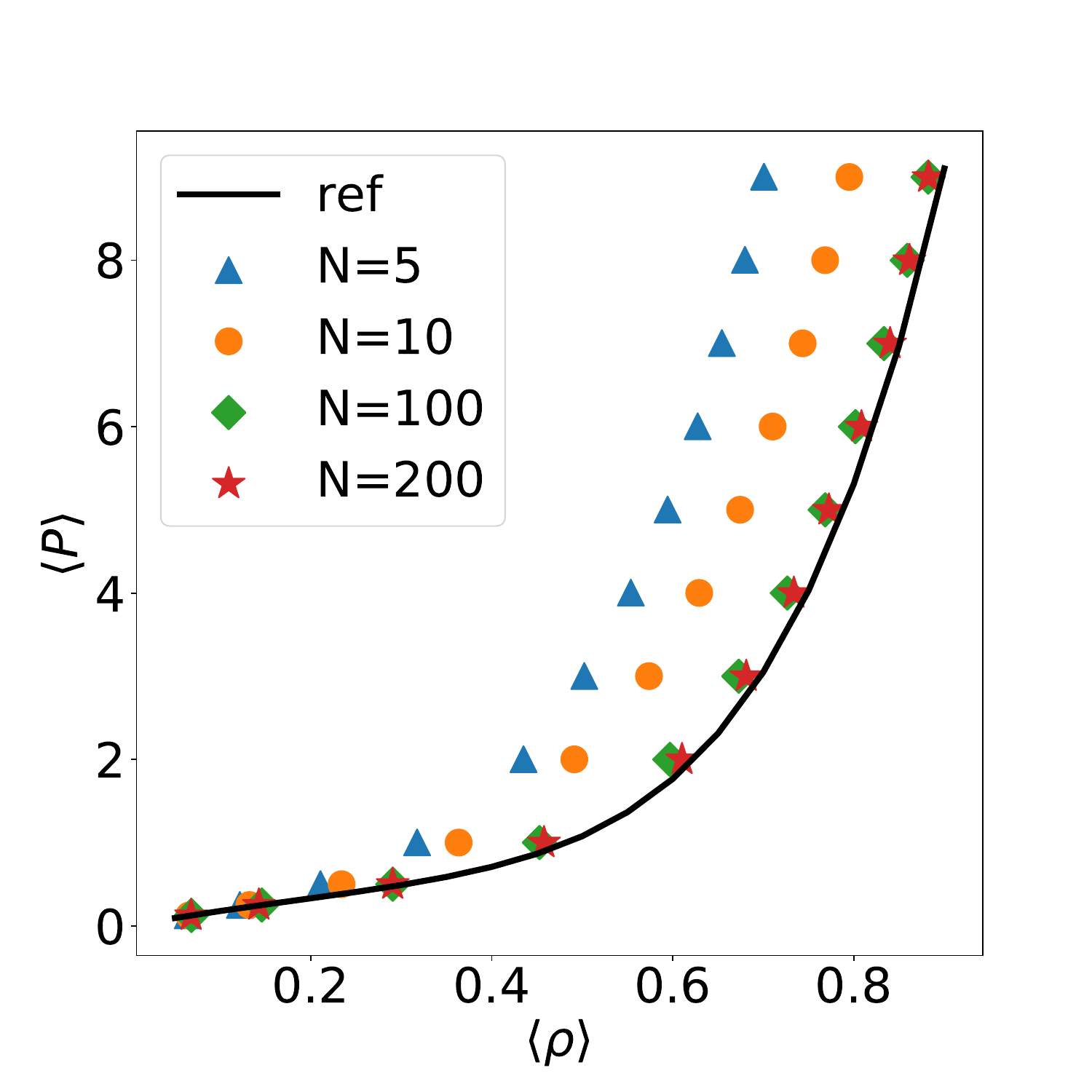}
    \caption{The ensemble average of pressure $\langle P\rangle$ verse the ensemble average of density $\langle\rho\rangle=\langle N/V\rangle$. The blue triangles, orange circles, green diamonds and red stars represent the numerical results of the systems with $N=5, 10, 100$ and $200$ particles, respectively. The results of $N=100$ and $N=200$ particles are closely aligned with the reference solution given by the fitting curve in \cite{johnson1993lennard} (blue curve), indicating that the relation between $\langle P\rangle$ and $\langle \rho\rangle$ converges as $N\rightarrow\infty$. }
    \label{fig:P-rho}
\end{figure}

In our experiment, simulations of system with different number of particles $N = 5, 10, 100$ are performed under fixed temperature $k_BT_0=2.0$ and various pressure $P_0$. All the particles initially located at a uniform lattice within the simulation box. $10^8$ iterations are performed with a time step of $\Delta t=10^{-5}$ in all the simulations. Relative errors in $\langle P\rangle$ and $\langle PV\rangle - P_0\langle V\rangle$ are denoted by 
\begin{gather}\label{eq:E1E2}
    E_1 = \frac{\left|\langle P\rangle - P_0\right|}{P_0} \quad\text{and}\quad E_2 = \beta \left|\langle PV\rangle - P_0\langle V\rangle + \beta^{-1}\right|,
\end{gather}
respectively. One can see from Table \ref{tab:LJ_error} and Figure \ref{fig:LJ_10_P-virial} (ensemble averages of the simulation with $N=100$ particles) that all the simulations capture the NPT ensemble well under different pressure $P_0$, which verifies the pressure virial theorems Eq. \eqref{eq:P-virial} and therefore comfirms the validity of the NPT ensemble. Figure \ref{fig:P-rho} shows the ensemble average of pressure, $\langle P\rangle$, under various ensemble average of density, $\langle \rho\rangle:=N/\langle V\rangle$, from which one can see that the numerical results of $N=100$ and $N=200$ particles (green diamonds and red stars, respectively) capture the reference solution (black curve) given by the fitting curve provided in \cite{johnson1993lennard} well, indicating that the relationship between $\langle P\rangle$ and $\langle \rho\rangle$ converges to the reference solution as $N\rightarrow \infty$. 
Taking the system with $N=10$ particles as an example, we perform 51 simulations at constant volume, with volume equally spaced in the range $[10, 60]$, to obtain the reference solution of the marginal distribution concerning $V$:
\begin{gather}\label{eq:rhoV_LJ}
    \rho_{V,\text{ref}}(V) \propto e^{-\beta\hat{F}(V)},
\end{gather}
where the free energy is computed by thermodynamic integration as \cite{bernetti2020pressure}
\begin{gather}
    \hat{F}(V+\Delta V) = \hat{F}(V) - \frac{\hat{P}(V) + \hat{P}(V+\Delta V)-P_0(V) - P_0(V+\Delta V)}{2}\Delta V
\end{gather}
with $\Delta V=1$. The left panel of Figure \ref{fig:rhoV_LJ} presents the marginal distribution concerning $V$, from which one can see that the empirical density obtained by $10^8$ iterations with time step $\Delta t=10^{-5}$ (red circles) is closely aligned with the reference one (blue curve). 

For comparison, taking the experiment of $N=100$ and $P_0=9.0$ as an example, we fix the time step at $\Delta t = 10^{-5}$ and perform $10^7$ iterations to obtain numerical marginal distribution with respect to $V$ using both the Trotter method provided in \cite[Appendix \uppercase\expandafter{\romannumeral5}.C]{bernetti2020pressure} and Algorithm \ref{alg:2ndLangevinSampler}. The reference solution is provided by the Trotter method, which is obtained using $10^8$ iterations with a time step of $\Delta t = 10^{-6}$. As shown in the right panel of Figure \ref{fig:rhoV_LJ}, the numerical solution obtained with Algorithm \ref{alg:2ndLangevinSampler} (blue squares) closely approximates the reference solution (black curve), whereas the solution obtained with the Trotter method (red triangles) deviates significantly, indicating that Algorithm \ref{alg:2ndLangevinSampler} converges faster than the Trotter method when using the same time step and evolution time. 

\begin{figure}[!htbp]
    \centering
    \subfigure{\includegraphics[width=0.495\textwidth]{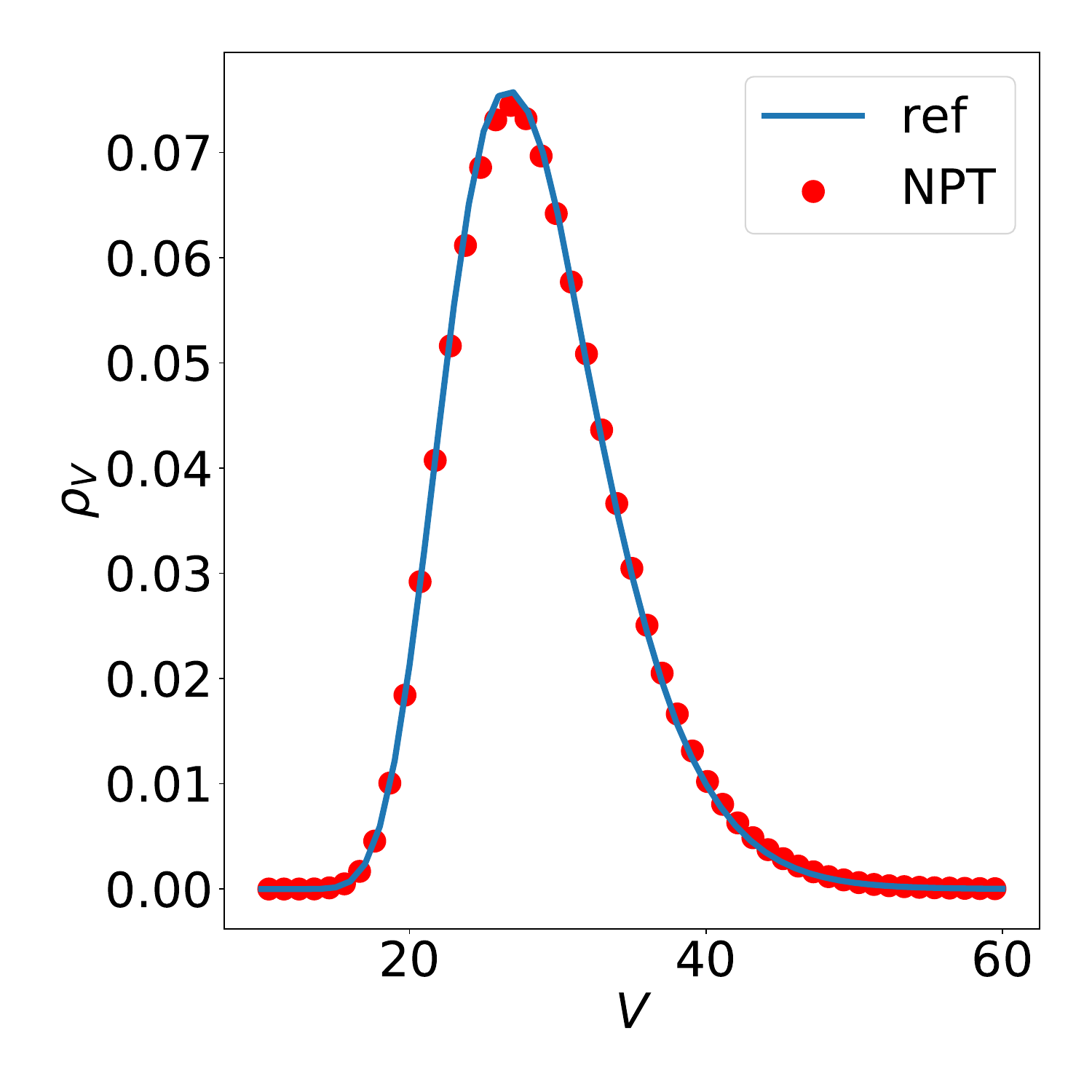}}
    \hfill
    \subfigure{\includegraphics[width=0.495\textwidth]{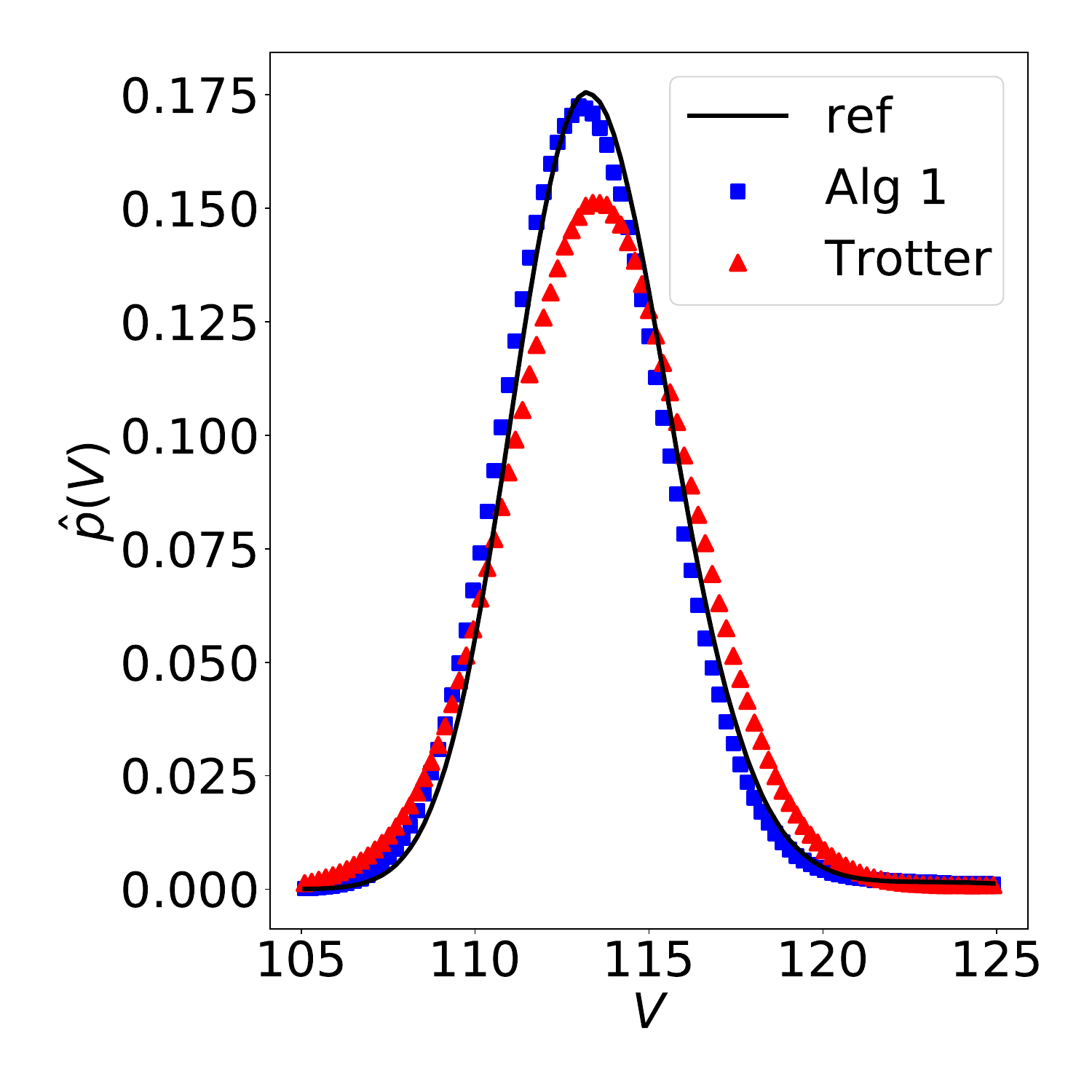}}
    \caption{Marginal distribution with respect to $V$. Left panel: The system contains $N=10$ particles, with fixed pressure $P_0 = 1.0$. The empirical density (red circles) is obtained after $10^8$ iterations with a time step of $\Delta t=10^{-5}$, and the reference density (blue curve) is given by Eq. \eqref{eq:rhoV_LJ}. The empirical density closely matches the reference density. Right panel: Comparison between the Trotter method and Algorithm \ref{alg:2ndLangevinSampler}. There are $100$ particles in the system, and the pressure is fixed at $P_0 = 9.0$. For both methods, the time step is fixed at $\Delta t = 10^{-5}$ with $10^7$ iterations. The numerical density obtained by Algorithm \ref{alg:2ndLangevinSampler} (blue squares) approximates the reference solution (black curve) more closely than the numerical density obtained by the Trotter method (red triangles).}
    \label{fig:rhoV_LJ}
\end{figure}

\section*{Acknowledgement}

This work was financially supported by the National Key R\&D Program of China, Project Number 2021YFA1002800.
The work of L. Li was partially supported by NSFC 12371400 and 12031013,  Shanghai Municipal Science and Technology Major Project 2021SHZDZX0102, and Shanghai Science and Technology Commission (Grant No. 21JC1403700, 21JC1402900).

\bibliographystyle{plain}
\bibliography{npt}

\end{document}